\documentclass[11pt,twoside]{amsart}
\usepackage{amssymb}

\usepackage[latin1]{inputenc}
\usepackage[T1]{fontenc}
\usepackage{mathtools}
\usepackage{graphicx}
\usepackage{tikz}
\usepackage{mathabx}
\usetikzlibrary{chains}
\usepackage[OT2,T1]{fontenc}

\PassOptionsToPackage{pdfusetitle,pagebackref,colorlinks}{hyperref}
\usepackage{bookmark}
\hypersetup{
  linkcolor={red!70!black},
  citecolor={green!70!black},
  urlcolor={blue!80!black}
}

\usepackage[latin1]{inputenc}
\usepackage{amsmath}
\usepackage{amsthm}

\usepackage[all]{xy}
\usepackage{hyperref}
\newtheorem{thm}{Theorem}[section]

\newtheorem{prop}[thm]{Proposition}

\newtheorem{lem}[thm]{Lemma}
\newtheorem{cor}[thm]{Corollary}
\newtheorem{thmx}{Theorem}
\numberwithin{equation}{section}

\theoremstyle{definition}

\newtheorem{remark}[thm]{Remark}

\newtheorem{ex}[thm]{Example}

\DeclareFontFamily{U}{mathc}{}
\DeclareFontShape{U}{mathc}{m}{it}%
{<->s*[1.03] mathc10}{}
\DeclareMathAlphabet{\mathcal}{U}{mathc}{m}{it}

\newcommand{\im}{\operatorname{im}}

\newcommand{\Aut}{{\rm Aut}}

\newcommand{\Br}{{\rm Br}}
\newcommand{\SBr}{{\rm SBr}}

\newcommand{\Pic}{{\rm Pic}}

\newcommand{\rk}{{\rm rk}}
\newcommand{\per}{{\rm per}}
\newcommand{\ind}{{\rm ind}}

\newcommand{\Hom}{{\rm Hom}}
\newcommand{\Spec}{{\rm Spec}}

\newcommand{\PGL}{{\rm PGL}}
\newcommand{\GL}{{\rm GL}}
\newcommand{\SL}{{\rm SL}}

\newcommand{\coker}{{\rm coker}}

\newcommand{\cal}{\mathcal}
\newcommand{\ka}{{\cal A}}

\newcommand{\ke}{{\cal E}}

\newcommand{\kg}{{\cal G}}
\newcommand{\kh}{H}

\newcommand{\kl}{{\cal L}}

\newcommand{\ko}{{\cal O}}
\newcommand{\kp}{{ Q}}

\newcommand{\kt}{{B}}

\newcommand{\kx}{{\cal X}}

\newcommand{\kz}{{\cal Z}}
\newcommand{\GG}{\mathbb{G}}

\newcommand{\ZZ}{\mathbb{Z}}
\newcommand{\QQ}{\mathbb{Q}}

\newcommand{\CC}{\mathbb{C}}

\newcommand{\PP}{\mathbb{P}}

\DeclareSymbolFont{cyrletters}{OT2}{wncyr}{m}{n}
\DeclareMathSymbol{\Sha}{\mathalpha}{cyrletters}{"58}

\renewcommand{\to}{\xymatrix@1@=15pt{\ar[r]&}}
\newcommand{\lto}{\xymatrix@1@=15pt{&\ar[l]}}
\renewcommand{\rightarrow}{\xymatrix@1@=15pt{\ar[r]&}}
\renewcommand{\mapsto}{\xymatrix@1@=15pt{\ar@{|->}[r]&}}
\newcommand{\mapslto}{\xymatrix@1@=15pt{&\ar@{|->}[l]&}}
\renewcommand{\twoheadrightarrow}{\xymatrix@1@=18pt{\ar@{->>}[r]&}}
\renewcommand{\hookrightarrow}{\xymatrix@1@=15pt{\ar@{^(->}[r]&}}
\newcommand{\hook}{\xymatrix@1@=15pt{\ar@{^(->}[r]&}}
\newcommand{\congpf}{\xymatrix@1@=15pt{\ar[r]^-\sim&}}
\renewcommand{\cong}{\simeq}

\newcommand{\TBC}[1]{}
\makeatletter
\def\blfootnote{\xdef\@thefnmark{}\@footnotetext}
\makeatother

\begin{document}

\title{Universal Brauer--Severi varieties}

\author[F.\ Gounelas, D.\ Huybrechts]{Frank Gounelas and Daniel Huybrechts }

\address{Mathematisches Institut \& Hausdorff Center for Mathematics,
Universit{\"a}t Bonn, Endenicher Allee 60, 53115 Bonn, Germany}
\email{gounelas@math.uni-bonn.de, huybrech@math.uni-bonn.de}

\begin{abstract} \noindent We construct universal Brauer--Severi varieties of fixed period and index and study their
geometry. We determine their cohomology and their Brauer and Picard groups and show that they are almost
always simply connected. As an application, we
reinterpret the discriminant avoidance
result of de Jong and Starr in terms of  universal Brauer--Severi varieties. 
 \vspace{-2mm}
\end{abstract}
%\dedicatory{ ERC Synergy Grant HyperK (ID 854361).}

\maketitle
\blfootnote{Both authors are supported by the ERC Synergy Grant HyperK (ID 854361).}
%%%%%%%%%%%%%%%%%%%

The purpose of this article is to give a concrete
 geometric meaning to certain constructions in the literature. The main motivation is to eventually be able to use global methods
in algebraic geometry to approach algebraic problems concerning Brauer groups of function fields.\smallskip

More concretely, we construct universal Brauer--Severi varieties $\kp\to\kt$ depending on
a fixed degree $n$ and period $d$ such that any other  Brauer--Severi variety with the same invariants
is obtained by pull-back. We phrase this here as follows, see later for more precise versions.

\begin{thmx}\label{thm1}  Fix three integers $m$, $n$, and $d$ with $d\mid n\mid d^e$ for some $e$.
Then there exists a Brauer--Severi variety $\kp\to\kt$ of relative dimension $n-1$, index $n$, and period $d$ over
a simply connected, quasi-projective, smooth variety $\kt$ such that
$$\Pic(\kt)\cong \ZZ~\text{ and }~\Br(\kt)\cong \ZZ/d\ZZ$$
and with the following universality property:
If $P\to X$ is any Brauer--Severi variety of relative dimension $n-1$ and  period dividing $d$
over a quasi-projective variety $X$ of dimension $\leq m$,
then $P$ is isomorphic to the pull-back of $\kp$ under a classifying morphism
$X\to \kt$.
\end{thmx}

All varieties will be considered over a fixed field $k$. The universality of our construction holds without
any further assumption, but the description of the fundamental group and of the Brauer group
assumes $k$ to be algebraically closed. \smallskip

Note that in the above statement,
the classifying map $X\to \kt$ is not necessarily unique. Also, there is in fact an ascending sequence
of universal Brauer--Severi varieties $\kp_0\subset \kp_1\subset\cdots$ over a sequence
of closed immersions $\kt_0\subset \kt_1\subset \cdots$. The above result is valid for any
$\kp_i\to\kt_i$ with $i\gg0$. For the minimal universal Brauer--Severi variety $\kp_0\to\kt_0$
the fundamental group is non-trivial and the Picard group is cyclic of finite order.

\begin{thmx}\label{thm2}
For $k=\CC$ the singular cohomology of $\kt_i$ is given
as 
$$\xymatrix{H^\ast(\kt_i,\QQ)\cong \bigwedge^\ast\langle \xi_{2n+1},\ldots,\xi_{2N-1}\rangle\otimes H^*({\rm Gr}(N,N+i),\QQ).}$$
Here, $N={n+d-1\choose n-1}$ and $\xi_{2j-1}\in H^{2j-1}(\kt_i,\QQ)$ are classes of weight $2j$ and type $(j,j)$ in the mixed
Hodge structure of $B_i$.
\end{thmx}

Pulling back the classes $\xi_{2j-1}$ under the classifying morphism $\psi\colon X\to \kt_i$ in Theorem \ref{thm1}
leads to classes in $\psi^\ast\xi_{2j-1}\in H^{2j-1}(X,\QQ)$ which we view as obstructions to extend the Brauer--Severi variety
$P\to X$. More precisely, if $P\to X$ extends to a Brauer--Severi variety $\bar P\to\bar X$
over a smooth complex projective variety $\bar X$, then $\psi^\ast(\xi_{2j-1})=0$, cf.\ Remark \ref{rem:obs}. Ultimately, the computation describes the cohomology of the classifying space of the group that governs Brauer--Severi varieties of relative dimension $n-1$ with a fixed relative $\ko(d)$,
cf.\ Remarks \ref{rem:AutO(d)} \& \ref{rem:class}. However, for geometric considerations the finite approximations $\kt_i$ are crucial.

\smallskip
The idea to use universal structures in the algebraic theory of central division algebras is not new,
but the geometric emphasis on Brauer--Severi varieties is. For example, Amitsur \cite{Amitsur} constructed the universal division algebra ${\rm UD}(n,k)$ of index $n$ 
over a certain field extension of $k$. The notion has been studied intensely in the algebra literature,
see \cite{AuelBr} for a recent survey. At least morally, there is a link between the universal
Brauer--Severi variety constructed here and ${\rm UD}(n,k)$. More precisely,
the division algebra associated to the generic fibre $\kp_\eta$ over the function field
$K(\kt)$ is obtained by specialisation of ${\rm UD}(n,k)$. However, we are not aware of a geometric
meaningful interpretation of this link. Variants of the universal division algebra taking  into account
also the period have been introduced and studied by Saltman \cite{SaltIndecomp,SaltmanBrauer}, see also \cite{Tignol}.
%Also note that our construction is finer,
%as it not only depends on the degree $n$, which is the index of $[\kp_\eta]\in\Br(K(\kt))$, but 
%also on the period $d$.
%There are only a few places in the literature where ${\rm UD}(n,k)$ has been studied by global techniques, t
The construction of $\kp\to\kt$ lends itself naturally to global considerations.
 For example, the $B_i$, $i\gg0$, admit compactifications by projective varieties with a boundary of arbitrarily high codimension which suggests the use of cohomological methods. Theorem \ref{thm2} is a first step in this direction and we plan to come back to this in a future paper.
 In another direction, various papers in the literature, e.g.\ \cite{first} construct $G$-torsors with similar versality properties for various algebraic groups $G$ (see Remark \ref{rem:first}). \smallskip

To demonstrate the usefulness of the notion of universal Brauer--Severi varieties for geometric
problems, we revisit  in \S\! \ref{sec:dJS} work of de Jong and Starr. Their article \cite{dJS}
was in fact the starting point for this paper. To provide some background, recall that the period $\per(\alpha)$ 
of a Brauer class $\alpha$ is simply its order and the index $\ind(\alpha)$ is the degree of the 
central division algebra representing $\alpha$ (at the generic point). It is known classically that
$\per(\alpha)\mid\ind(\alpha)\mid\per(\alpha)^{e_\alpha}$ for some $e_\alpha$ and Colliot-Th\'el\`ene
 \cite{CT} conjectured that one can choose $e_\alpha=\dim(X)-1$ for all $\alpha\in\Br(K(X))$ with $X$ a variety
over an algebraically closed field. The conjecture has been proved for unramified classes on
projective surfaces by de Jong \cite{dJ} with further contributions by de Jong and Starr \cite{dJS}
and Lieblich \cite{Lieb2,Lieb1}. Despite some recent results for unramified classes \cite{dJP,HM}, the conjecture is essentially completely open in higher dimensions. However, de Jong and Starr managed to prove that once the conjecture is settled for unramified classes,
so for $\alpha\in\Br(X)$ with $X$ projective, it also holds for all ramified classes $\alpha\in\Br(K(X))$.
Their result, not explicitly stated in \cite{dJS}, runs under the name of `discriminant avoidance' and
will be recalled as Theorem \ref{thm}. We provide a proof that roughly follows the original arguments
in \cite{dJS} but in the language of universal Brauer--Severi varieties.\smallskip

\medskip

\noindent
{\bf Acknowledgement:} The text grew out of a talk in a seminar on Brauer groups in  2025. The original goal was to present the discriminant avoidance
result of de Jong and Starr \cite{dJS}. Looking for alternative ways to present the material, we came
up with the notion of the universal Brauer--Severi variety which not only sheds a different light
on the results in \cite{dJS} but turned out to be of independent interest.
We wish to thank the participants of the seminar for interesting questions and feedback.
Special thanks to Giacomo Mezzedimi for the organisation of the seminar. We are also grateful
to Nicolas Perrin for answering a naive question on reductive group quotients, to Asher Auel for helpful comments on a preliminary version and his generous help with the literature, and to Uriya First for bringing \cite{first} to our attention and for a
helpful conversation.

%Olivier Benoist, Mezzedimi for running the seminar, Martin Bright, ...

%\medskip
%Our paper mainly focuses on the construction of the universal Brauer--Severi varieties $\kp_i\to\kt_i$
%for given period $d$ and index $n$ and their properties. We address the existence of
%small compactifications, the description of basic invariants like fundamental, Picard, and Brauer group,
%and look into geometric features. A quick look at the list of contents will give a good impression of what is covered.
\tableofcontents
%%%%%%%%%%%%%%%%%%%
\section{Universal Brauer--Severi varieties}
In this section we introduce the universal quasi-projective Brauer--Severi variety $Q\to B$ of index $n$ and period $d$. In fact, we construct a series $Q_i\to B_i$, $i\geq 0$, of universal Brauer--Severi varieties contained
in each other and one should think of the limit $Q_\infty=\lim Q_i\to B_\infty=\lim B_i$ as the infinite-dimensional universal
Brauer--Severi variety. The distinguishing property of $Q_i\to B_i$ for large $i$ is the existence
of a compactification by a projective variety with a boundary of arbitrarily large codimension. Furthermore,
various (cohomological) invariants will turn out to stabilise for large $i$, see \S\! \ref{sec:Coh}.
\smallskip

 The construction is motivated by a result of de Jong and Starr \cite[Prop.\ 2.5.1]{dJS} for general reductive groups. Our construction only deals with Brauer--Severi varieties, so the group $\PGL$,
 but is more  concrete and potentially more transparent from a geometric point of view.

\subsection{Preparations: period and index}

Assume $\pi\colon P\to X$ is a Brauer--Severi variety of relative dimension $n-1$ over a smooth quasi-projective variety $X$ over a %n algebraically closed
field $k$ and let  $\alpha=[P]\in \Br(X)$ be its Brauer class.\smallskip

 For the convenience of the reader, we recall two well-known facts
that allow for a geometric way of computing the period and the index
of a Brauer class. We start with the period, cf.\ \cite[\S\! 2]{Artin}.

\begin{lem}\label{lem:Or}
The period $\per(\alpha)$ divides an integer $d$ if and only if there exists a line bundle $\ko(d)$
on $P$ which when restricted to every geometric fibre $P_x\cong\PP^{n-1}$  gives
$\ko(d)$. In other words,
$$\per(\alpha)=\min\{\,d\mid \exists\,\, %\text{there exists a line bundle }
 \ko(d)\in\Pic(P)\,\}.$$
\end{lem}

\begin{proof}
%Since $X$ is assumed to be smooth, its Brauer group injects into the Brauer group of its function field $K(X)$. If
For example, one can argue using twisted sheaves. If $P$ is viewed as the
projectivisation of  an $\alpha$-twisted locally free sheaf $E$, then the direct image of a line bundle $\ko(d)$ 
is untwisted and, up to tensoring with line bundles, isomorphic to the dual of the symmetric
product $S^dE$. Hence,  $S^dE$ is untwisted and, therefore, $\alpha^d=1$. Conversely,
combining the fact that $\ko(1)$ always exists as a $\pi^\ast\alpha$-twisted sheaf and 
that $\alpha$ can be represented by a cocycle with $\alpha_{ijk}^d=1$ for $d=\per(\alpha)$,
we find that $\ko(d)$ always exists as an untwisted sheaf.
\end{proof}

Note that the dual of the relative canonical bundle $\omega_\pi$ is fibrewise  just $\ko(-n)$,
which implies the well-known divisibility $\per(\alpha)\mid n$ and, hence,  $\per(\alpha)\mid \ind(\alpha)$.\smallskip

\begin{remark}\label{rem:AutO(d)}
The cohomological version of the lemma uses the short
exact sequence of sheaves in the \'etale topology (or fppf when the characteristic divides $n$):
%$$1\to \GG_m\to\Aut(\PP^{n-1},\ko(r))\to \Aut(\PP^{n-1})\cong\PGL(n)\to 1,$$
\begin{equation}\label{eqn:GmGLPGL}
\xymatrix@R=0.1pt{1\ar[r]& \GG_m\ar[r]&\Aut(\PP^{n-1},\ko(d))\ar[r]& \Aut(\PP^{n-1})\ar[r]& 1.\\
%& &&\cong\PGL(n)&&
}
\end{equation}
Here, $ \Aut(\PP^{n-1})\cong\PGL(n)$ and $\Aut(\PP^{n-1},\ko(d))$ is the group of
automorphisms of the line bundle $\ko(d)$ over $\PP^{n-1}$. For $d=1$ the latter is simply
$\GL(n)$ and (\ref{eqn:GmGLPGL}) is the short exact sequence defining $\PGL(n)$. For $d>1$ one
finds $\Aut(\PP^{n-1},\ko(d))\cong\GL(n)\,/\, \mu_d$ and the special case $d=n$ leads to
the split group $\Aut(\PP^{n-1},\ko(n))\cong \GG_m\times\PGL(n)$. To see this, use that any automorphism of $\PP^{n-1}$ lifts naturally to an  automorphism of the line bundle $\ko(n)\cong\omega^\ast_{\PP^{n-1}}$ or
that $\det\colon \GL(n)\to \GG_m$ induces a splitting of $\GG_m\cong\GG_m\,/\,\mu_d\,\hookrightarrow \GL(n)\,/\,\mu_d$.
%Note that for $d\mid n$ one also has $\GL(n)\,/\, \mu_d\cong\GL(n)$.
Cohomological properties of $\SL(n)\,/\,\mu_d$ have been investigated in the past, see e.g.\ \cite[\S\! 4]{BR}.\smallskip

Furthermore, sequence (\ref{eqn:GmGLPGL}) induces  the exact cohomology sequence
$$\xymatrix{H^1(X,\Aut(\PP^{n-1},\ko(d)))\ar[r]& H^1(X,\PGL(n))\ar[r]^-\delta&H^2(X,\GG_m)\ar[r]&}$$
and we have that  $\delta[P]=d\cdot\alpha$, cf.\ \cite[Ex.\ 1.1]{Merk2}.
%Note that for $d=n$ one has $\Aut(\PP^{n-1},\ko(n))\cong \GL(n)$.
\end{remark}

The Brauer--Severi variety $P\to X$ restricted to any open subset $U\subset X$ defines a Brauer--Severi variety $P_U\to U$ over $U$. Similarly, restricting to the generic point $\eta=\Spec(K(X))\in X$ defines a Brauer--Severi variety
$P_\eta$ over the function field $K(X)$ of $X$. This induces the
natural homomorphism $\Br(X)\to\Br(K(X))$, $\alpha\mapsto\alpha_\eta$, which is injective for locally factorial $X$, cf.\ \cite[Thm.\ 3.5.5]{CTS}. Of course, every class in $\Br(K(X))$ is realised by a Brauer class
and by a Brauer--Severi variety on some open subset $U\subset X$ and, in order to argue geometrically, we often fix such a realisation.

\medskip

Note that in general the index $\ind(\alpha)$ of a Brauer class $\alpha\in \Br(X)$ is not realised as the degree $\deg(\ka)$ of an Azumaya algebra $\ka$ representing $\alpha$, see \cite{AW}. Equivalently, $\ind(\alpha)$ can typically not be computed as $\dim(P\to X)+1$ of any single  Brauer--Severi variety $P\to X$  representing $\alpha$. This is different for Brauer classes over the function field. Indeed,
 the index of a Brauer class over a field is the degree of the unique central division algebra realising the class. Also recall that period and index behave well under the restriction morphisms, i.e.\
$\ind(\alpha)=\ind( \alpha_U)=\ind(\alpha_\eta)$.\smallskip

For the proof of the following we refer to \cite[\S\! 3]{Artin} or \cite[\S\! 8]{Kollar}.

\begin{lem} 
Assume $\alpha\in \Br(K)$ is a Brauer class over a field $K$ represented by a Brauer--Severi variety $P$ over $K$. Then
$$\ind(\alpha)=\min\{\,\dim(P')+1\mid P'\subset P\text{ linear}\,\}.$$
\vskip-0.8cm\qed
\end{lem}
The minimum on the right hand side is taken over all $K$-rational linear subspaces. For example, if $P$ admits a rational point, which we may view as a zero-dimensional linear subspace $P'\subset P$, then
$\ind(\alpha)=1$, i.e.\ $\alpha\in \Br(K)$ is in fact trivial. Also, by definition of the index, we know that $\ind(\alpha)\mid \dim(P)+1$, which restricts the possible dimensions of linear subspaces. For example, $P$ does not contain a codimension 
one linear subspace unless $\alpha$ is trivial, which fits well with Lemma
\ref{lem:Or}.

\begin{remark}
For a Brauer--Severi variety $P\to X$ over a locally factorial variety $X$ the lemma can be rephrased geometrically by saying that the index is the minimum of all $n'$ such that the relative Grassmannian $\text{Gr}(n'-1,P)\to X$ admits a rational section, something we will explore in \S\! \ref{sec:dJS}.
\end{remark}

%%%%%%%%%%%%%%%%%%%

\subsection{Construction}Let $V$ be a vector space of dimension $n$ over a field $k$ and let
$d$  and $i$ be non-negative integers. In the applications later, we will assume that $d$ divides $n$ and 
that $n\mid d^e$ for some exponent $e$. One could also just avoid choosing $d$ by setting $d=n$, but the results would be less precise.\smallskip

We introduce the vector spaces $S^d_iV\coloneqq S^dV\oplus k^{\oplus i}$
which are of dimension $N_i=N_i(d,n)={n+d-1\choose n-1}+i$. We often simply write $N=N_0$.
The natural inclusions 
\begin{equation}\label{eqn:SdVi}
S^dV=S^d_0V\subset S^d_1V\subset\cdots
\end{equation}
are all of codimension one and we will write
$$\nu_{d,i}\colon \PP(V)\,\hookrightarrow \PP(S^dV)\,\hookrightarrow \PP(S^d_iV)$$
for the induced composition with the $d$-th Veronese embedding $\nu_d=\nu_{d,0}$. We recommend to restrict to the case $i=0$ during a first reading.\smallskip

The image of each of the $\nu_{d,i}$ defines a $k$-rational point
$$[\im(\nu_{d,i})]\in\kh_{d,i}\coloneqq{\rm Hilb}(\PP(S^d_iV))$$
in the Hilbert scheme of closed subschemes of $\PP(S^d_iV)$.\smallskip

The Hilbert scheme $\kh_{d,i}$ comes with the natural action of $\PGL(S^d_iV)\cong\Aut(\PP(S^d_iV))$.
The stabiliser of this action at the point $[\im(\nu_{d,i})]\in \kh_{d,i}$ shall be denoted
 $\PGL(V,i)$. In other words, $\PGL(V,i)\subset \PGL(S^d_iV)$ is the subgroup of all elements that restrict to an automorphism of $\PP(S^dV)$, which then is induced by an element in
 $\PGL(V)$  via the natural embedding $\PGL(V)\,\hookrightarrow \PGL(S^dV)$.
 In concrete terms,  $\PGL(V,i)$ is the quotient $\GL(V,i)/k^\ast$ with
 $$ \GL(V,i)=\left(\begin{matrix}\GL(V)&\ast\\
0&\GL(k^{\oplus i})\end{matrix}\right).$$

  Hence, the orbit $\kt_i\subset \kh_{d,i}$ through $[\im(\nu_{d,i})]\in \kh_{d,i}$
 is isomorphic to the quotient 
 \begin{eqnarray}\label{eqn:kt}\kt_i\coloneqq\kt_i(d,n)&\coloneqq&\PGL(S^d_iV)\cdot[\im(\nu_{d,i})]\\
 &\cong&\PGL(S^d_iV)\,/\,\PGL(V,i)\cong\GL(S^d_iV)\,/\,\GL(V,i).\nonumber
 \end{eqnarray}
 Note that the last quotient is not a quotient by a subgroup, for 
 $\GL(V)\to\GL(S^dV_i)$ has a non-trivial kernel isomorphic to $\mu_d$.
The description of $\kt_i$ as a quotient allows one to easily compute its dimension as
 \begin{equation}
 \dim\kt_i=N^2+iN-n^2,
 \end{equation}
which only depends on the `index' $n$, the `period' $d$, and the additional integer $i$.
\smallskip

It will turn out that certain features distinguish $$\kt_0\cong\GL(S^dV)\,/\,\GL(V)$$ from the higher $\kt_i$, $i>0$. For one, $\kt_0$ is affine, see  \S\! \ref{sec:comp} below, while the higher ones $\kt_1,\kt_2,\ldots$ are not.
\smallskip

By construction, $\kt_i\subset\kh_{d,i}$ is the open subset of the Hilbert scheme that parametrises all closed subschemes of $\PP(S^d_iV)$
that are projective equivalent to $\im(\nu_{d,i})\subset\PP(S^d_iV)$ over the algebraic closure of $k$.
The universal family of those is obtained by restricting the universal subscheme $\kz\subset
\kh_{d,i}\times\PP(S^d_iV)$ to $\kt_i\subset\kh_{d,i}$. We call it the \emph{universal Brauer--Severi variety}
$$\pi_i\colon\kp_i\coloneqq\kz_{\kt_i}\to\kt_i.$$

As a homogeneous space, $\kp_i$ is described as the quotient
\begin{equation}\label{eqn:kphomog}
\kp_i\cong\PGL(S^d_iV)\,/\,\PGL(V,i)_x\cong\GL(S_i^dV)\,/\,\GL(V,i)_x,
\end{equation}
and as a homogeneous bundle it can be written as
$$\kp_i\cong\PGL(S^d_iV)\times_{\PGL(V)}\PP(V).$$
Here, $\PGL(V,i)_x\subset\PGL(V,i)$ is the stabiliser of a point $x\in \PP(V)$ of the natural action of $\PGL(V,i)$ on $\PP(V)$ and, hence, $\PGL(V,i)_x\cong\GL(V,i)_x\,/\,k^\ast $ with
\begin{equation}\label{eqn:GLVx}
\GL(V,i)_x\cong\left(\begin{matrix}\left(\begin{matrix}k^\ast&\ast\\0&\GL(n-1)\end{matrix}\right)&\ast\\0&\GL(k^{\oplus i})\end{matrix}\right).
\end{equation}

Note that by construction, the universal Brauer--Severi variety $\kp_i$ is homogeneous with respect to the action of $\PGL(S^d_iV)$ and the projection $\kp_i\to\kt_i$ is equivariant.\smallskip

The closed embeddings  $\PP(S^d_iV)\subset \PP(S^d_{i+1}V)$ induce closed embeddings $\kt_i\subset\kt_{i+1}$ such that the restriction of $\kp_{i+1}|_{\kt_i}$ gives back $\kp_i$. In other words, there is a commutative diagram
$$\xymatrix{\kp_0\ar[d]_{\pi_0}\ar@{^(->}[r]&\kp_1\ar[d]_{\pi_1}\ar@{^(->}[r]&\ldots\ar@{^(->}[r]&\kp_i\ar[d]_{\pi_i}\ar@{^(->}[r]&\kp_{i+1}\ar[d]_{\pi_{i+1}}\ar@{^(->}[r]&\\
\kt_0\ar@{^(->}[r]&\kt_1\ar@{^(->}[r]&\ldots\ar@{^(->}[r]&\kt_i\ar@{^(->}[r]&\kt_{i+1}\ar@{^(->}[r]&}$$

%%%%%%%%%%%%%
\subsection{Compactifications}\label{sec:comp}
For $i=0$ the base $\kt_0$ is the quotient of the affine algebraic group
$\PGL(S^dV)$ by the reductive group $\PGL(V)$ and hence affine. In particular, any compactification
$\kt_0\subset\bar\kt_0$ by a projective variety $\bar\kt_0$ has at least one boundary component of codimension one. However, as  $i$ grows the situation changes.
To make this precise, we use arguments that go back to
 \cite[Rem.\ 1.4]{Totaro}, but see also \cite[Lem.\ 9.2]{CTSa}, \cite[\S\! 2.4]{dJS}, and \cite{EG}. Our original motivation comes from trying to understand \cite[Prop.\ 2.5.1]{dJS} 
in concrete geometric terms. Our varieties $\kt_i$ are the analogue of the smooth atlas $X\to{\rm BPGL}$ constructed by de Jong and Starr which comes with a small compactification $\bar X$ of $X$, such that every  $\Spec(K)\to{\rm BPGL}$  lifts to $X$.

\begin{prop}\label{prop:Dim} Let $c>0$ be a fixed integer and choose $i$ such that
$i\geq c-1$. Then, $\kt_i\cong\PGL(S^d_iV)\,/\,\PGL(V,i)$ admits
a projective compactification  all of whose boundary components
are of codimension $\geq c$.
\end{prop}

\begin{proof} The key is to reinterpret the quotient (\ref{eqn:kt}) 
as the quotient of a linear space by $\GL(V)$. For this we consider 
the space  $U\coloneqq\Hom(S^dV,S^d_iV)$ of all linear maps
$\psi\colon S^dV\to S^d_iV$ and denote by %$U_{<N}$ the closed subset of all linear maps that are not surjective. Clearly, its complement $U\, \setminus\,  U_{<N}$
$U_N\subset U$ the open set of all injective maps. Clearly, $U_N$ is  homogeneous under
the natural action of $\GL(S^d_iV)$. The stabiliser of the inclusion
$ S^dV\,\hookrightarrow S^d_iV=S^dV\oplus k^{\oplus i}$ is isomorphic to 
$$H\coloneqq \left(\begin{matrix}{\rm id}_{S^dV} &\ast\\
0&\GL(k^{\oplus i})
\end{matrix}\right)$$
and, therefore, $U_N\cong \GL(S^d_iV)\,/\,H$.
This quotient comes with a further right action of $\GL(S^dV)$ which is given by pre-composition.
We will only need the induced action of the smaller subgroup $\GL(V)\subset\GL(S^dV)$, where we abuse notation and denote by  $\GL(V)$ also its image in $\GL(S^dV)$. Since $\GL(V,i)=H\cdot\GL(V)$, this leads to the alternative description of $\kt_i$ as
$$\kt_i \cong \PGL(S^d_iV)\,/\,\PGL(V,i)\cong 
\GL(S^d_iV)\,/\,\GL(V,i)\cong U_N\,/\,\GL(V).$$
Clearly, the above right action of $\GL(V)$ on $U_N$  extends naturally to a right action
$U\times \GL(V)\to U$, again by pre-composition. The compactification we want to use is provided by the GIT quotient
$U^{\text ss}/\!/\,\GL(V)$.

\begin{lem}\label{lem:injimpliesss}
    We have $U_N\subseteq U^{\text ss}$.
    %\footnote{This should be seen as a classical result in GIT, but we could not find an appropriate reference. We are still debating whether to include a complete proof.}
\end{lem}

%\begin{proof} \end{proof}

\TBC{Need reference. The stabiliser in $\GL(S^dV)$ of
a surjection $W\twoheadrightarrow S^dV$ consists of scalars and,
hence, the stabiliser in $\PGL(V)\subset\PGL(S^dV)$ is trivial. To produce an invariant section of some $\ko(a)$, $a\gg0$, non-vanishing at a surjection $\psi\colon W\twoheadrightarrow S^dV$,
take $\bigwedge^N\psi\colon \bigwedge^NW\twoheadrightarrow \bigwedge^NS^dV\cong k$ and evaluate it at an element not contained
in the kernel. This produces an invariant section of $\ko(N)$ non-vanishing at $\psi$.
} Hence, $\kt_i$ can be indeed viewed as an open subset of the projective(!) GIT quotient
\begin{equation}\label{eqn:openemb}
\kt_i\subset \bar\kt_i\coloneqq U^{\text ss}/\!/\,\GL(V)=\PP(U)^{\text ss}/\!/\,\PGL(V).
\end{equation}
To control the dimension of its boundary $\partial \kt_i\subset\bar\kt_i$, one uses the stratification
$U\,\setminus \,U_N=\bigsqcup_{\ell=0}^{N-1}U_\ell$, where the strata are defined as $U_\ell\coloneqq\{\,\psi\mid \rk(\psi)=\ell\,\}$. The dimension of each stratum is, cf.\ \cite[\S\! 7.2.A]{lazarsfeldpositivityvol2}:
\begin{eqnarray*}
\dim(U_\ell)&=&\dim\Hom(S^dV,k^{\oplus \ell})+\dim{\rm Gr}(\ell,S^d_iV)
\\
&=&N\cdot\ell+\ell\cdot(N_i-\ell)=\ell\cdot(2N+i-\ell).
\end{eqnarray*}
%and hence its codimension is given by \begin{eqnarray*}{\rm codim}(U_\ell)&=&\dim(W)\cdot N-\ell\cdot(N-\ell)-\dim(W)\cdot\ell\\ &=&(\dim(W)-\ell)\cdot(N-\ell).
%\end{eqnarray*}
We claim that the codimension of the image of $U_\ell^{\rm ss}$ in $\bar\kt_i$ is at least $c$.
For $\ell\leq n$ we have in fact already $\dim(\kt_i)-\dim(U_\ell)\geq i+1\geq c$, which is clearly enough.
For $\ell\geq n$ we first observe that $U_\ell$ always contains a point with finite stabiliser.
Indeed, any $\psi\colon S^dV\to S^d_iV$ with $\psi(v_i^d)=v_i^d$ for $v_1,\ldots ,v_n\in V$ a basis,
has a stabiliser that is contained in $\mu_d^n$ and for $\ell\geq n$ every $U_\ell$ contains a $\psi$ with this property. This implies that the dimension of the image of $U_\ell^{\rm ss}$ in $\bar\kt_i$ can be computed
as $\dim(U_\ell)-n^2$. Thus, it suffices to show that $\dim(\kt_i)-\dim(U_\ell)+n^2\geq i+1\geq c$,
which is easily verified.
%which for $0\leq \ell\leq N-1$ is bounded from above by $\dim(U_{N-1})=(N_i+1)\cdot(N-1)$.\smallski%This leads to the desired lower bound for the codimension of the boundary
%$\partial\kt_i$ of the open embedding (\ref{eqn:openemb}). %$\kt_W\subset \PP(U)^{\text ss}/\!/\,\PGL(V)$. \TBC{The stabiliser of a non-surjective map $\psi\colon W\twoheadrightarrow S^dV$. If $A=\im(\psi)$ is its $\ell$-dimensional image, then this stabiliser is isomorphic to the subgroup $$\PGL(V)\cap \left(\begin{matrix}{\rm id}_{\PP(A)}&\ast\\ 0&\ast\end{matrix}\right)\subset\PGL(S^dV).$$ I am not completely sure how to compute its dimension}   For the purposes of this paper, the following crude estimate is enough (a more precise bound would require computing the stabilisers at points in $U_\ell$, $0\leq \ell\leq N-1$): Since $\dim (\kt_i)=\dim(\PP(U)^{\text ss}/\!/\,\PGL(V))=\dim(U)-n^2=N_i\cdot N-n^2$ and the dimension of the image of $U_\ell$ in $\PP(U)^{\text ss}/\!/\,\PGL(V)$ is bounded from above by $\max\{\dim(U_\ell)\}=\dim(U_{N-1})=(N_i+1)\cdot(N-1)$, we obtain
%\begin{eqnarray*}{\rm codim}(\partial \kt_i)&\geq&
%\dim(W)\cdot N-n^2-\max\{\dim (U_\ell)\}%\ell\cdot (N-\ell)-\dim(W)\cdot\ell
%\\&\geq& \dim(W)\cdot N-n^2-\dim (U_{N-1})\\&=&i+1-n^2\geq c \end{eqnarray*}
%Since by assumption the right side is bounded from below by $c$, t
\end{proof}

%The main difference between  $\kt_0$ and $\kt_i$ is that while $\kt_0$ is affine and hence any morphism $\varphi\colon X\to \kt_0$ from a proper $X$ is constant, the same is no longer true for $\kt_i$, $i\gg0$.  is usually not. %(That is the reason why in Proposition \ref{prop:localUniv} we had to restrict to open subsets $U_i$.)

The construction in the above proof, due to Totaro \cite{Totaro}, appears mysterious at first glance, but it can be given a geometric interpretation as follows.
The further quotient by $\GL(S^dV)$ leads to a morphism to the Grassmannian $$\kt_i\cong U_N\,/\,\GL(V)\twoheadrightarrow U_N\,/\,\GL(S^dV)\cong {\rm Gr}(N, S^d_iV),$$ which  maps $[Z]\in \kt_i\subset {\rm Hilb}(\PP(S^d_i))$ to its span $\PP^{N-1}\cong\langle Z\rangle\subset\PP(S^d_iV)$ viewed as a point in ${\rm Gr}_i\coloneqq{\rm Gr}(N, S^d_iV)$. Those fit into a commutative diagram of closed immersions
 $$\xymatrix{\kt_0\ar@{^(->}[r]\ar[d]_{p_0}&\kt_1\ar@{^(->}[r]\ar[d]_{p_1}&\kt_2\ar@{^(->}[r]\ar[d]_{p_2}&\kt_3\ar@{^(->}[r]\ar[d]_{p_3}&\cdots\\
 \text{Spec}(k)\ar@{^(->}[r]&\PP(S^d_1V^\vee)\ar@{^(->}[r]&{\rm Gr}_2\ar@{^(->}[r]&{\rm Gr}_3\ar@{^(->}[r]&\cdots,}$$
 where the embedding ${\rm Gr}_i={\rm Gr}(N,S^d_iV)\,\hookrightarrow {\rm Gr}_{i+1}={\rm Gr}(N,S^d_{i+1}V)$ is obtained
 by viewing an $N$-dimensional subspace $A\subset S^d_iV=S^dV\oplus k^{\oplus i}$ as an $N$-dimensional  subspace of $S^d_{i+1}V=S^dV\oplus k^{\oplus i}\oplus k$.

\begin{ex}\label{ex:codim2}
As a special case, we mention that $\kt_i$ admits a compactification with a boundary
of codimension at least two as soon as  $i\geq1$.
Note, however, the compactification $\kt_i\subset\bar\kt_i$ constructed above should be expected to be  singular (but all singularities are rational, cf.\ \cite{Boutot}).\smallskip

 Indeed, at least for $k=\CC$,
the quotient $\kt_i\cong U_N\,/\!/\,\GL(V)$ is known to be stably rational, see  \cite[Prop.\ 4.7]{CTSa}
or \cite[Thm.\ 1.1]{BoKa} for the quotient by $\SL(V)$.
If $\bar \kt_i$ were smooth, then by purity $\Br(\kt_i)\cong\Br(\bar\kt_i)$, but the 
Brauer group of a stably rational smooth projective variety is trivial and we show below that it is
not for $\kt_i$.
\end{ex}

\begin{ex}
The easiest case where we can describe $\kt_0$ fully is the universal conic, i.e.\ the case $n=d=2$ for which $N=3$. The relevant component of the Hilbert scheme is
the linear system $|\ko_{\PP^2}(2)|\cong\PP^5$ and $\kt_0\subset|\ko_{\PP^2}(2)|$ is simply the open subset of all smooth conics in $\PP^2$ or, in other words, $\kt_0$ is the complement of the discriminant
divisor which is a hypersurface $D\subset \PP^5$ of degree four. As such, $\kt_0$ is certainly affine
and $|\ko_{\PP^2}(2)|$ should count as its minimal compactification.\smallskip

Later we will be interested in various invariants of the spaces $\kt_i$, cf.\ Propositions \ref{prop:Funda},
\ref{prop:Picard}, and \ref{prop:BrTW}. We remark here already that for $n=d=2$ the group $\pi_1(\kt_0)$ cyclic of order three and not four, as one might expect for the complement of a degree four hypersurface, but $D$ is singular. This classical fact is surprisingly hard to find in the literature, but it is covered by the more general results in \cite{Loenne}.
\end{ex}

\begin{remark}\label{rem:codimBound}
In \S\! \ref{sec:covfam} we will consider the fibres $\kt_{i,x}$ of the projection $\kp_i\to\PP(S^d_iV)$
over a point $x\in\PP(S^d_iV)$ as a closed subscheme of $\kt_i$. The arguments in the above proof
can be adapted to prove that for $i\gg0$ also $\kt_{i,x}$ admits a projective compactification with boundary components of arbitrarily high codimension. In fact, it is the closure of $\kt_{i,x}$ in the GIT quotient $\PP(U)^{\rm ss}\,/\!/\,\PGL(V)$ that provides the desired compactification.\TBC{TBC}
\end{remark}
%%%%%%%%%%%%%%%%%%

\subsection{Classifying maps I}\label{sec:ClassmapsI}
 The universal Brauer--Severi varieties serve various purposes.
For those problems that are concerned with the ramified Brauer group, so the Brauer group $\Br(K(X))$ of the function
field of a variety, the case $i=0$ is the most natural one. For more global aspects, the higher
$\kp_i\to\kt_i$, $i\gg0$, are more useful.\smallskip

 In the following, we fix again $n$ and $d$ and
let $\kp_0\to\kt_0$ be the $0$th universal Brauer--Severi variety constructed before.\smallskip

%We shall first consider the case $i=0$ and in order to simplify the notation, we will
%drop the index $i=0$ in the universal Brauer--Severi variety and simply write $\kp\to\kt$ instead
%of $\kp_0\to \kt_0$.

\begin{prop}\label{prop:localUniv}
Assume $\pi\colon P\to X$ is a Brauer--Severi variety (with $X$ an arbitrary variety) of relative dimension $n-1$ with its period dividing the fixed integer $d$. Then there exists a Zariski open covering
$X=\bigcup U_i$ and morphisms $\varphi_i\colon U_i\to \kt_0$ such that
the restriction $P|_{U_i}\to U_i$ is isomorphic to the pull-back of $\kp_0\to \kt_0$ under $\varphi_i$, i.e.\
there exist isomorphisms
$$P|_{U_i}\cong\varphi_i^\ast\kp_0\cong \kp_0\times_{\kt_0} U_i$$ that are compatible with the projections onto $U_i$.
\end{prop}

\begin{proof} By virtue of Lemma \ref{lem:Or}, the period of $P$ divides $d$ if and only if there exists a line bundle
 on $P$ which fibrewise is $\ko(d)$ on $\PP^{n-1}$. The dual of its (untwisted!) direct image $\ke\coloneqq \pi_\ast\ko(d)^\vee$ is locally free of rank $N$
and the relative Veronese map defines a closed embedding
\begin{equation}\label{eqn:relVer1}
\xymatrix@C=10pt@R=15pt{P~\ar[dr]\ar@{^(->}[rr]&&\PP(\ke)\ar[dl]\\
&X&}
\end{equation}
relative over $X$.

Now pick a Zariski open covering $X=\bigcup U_i$ such that the restriction
$\ke_{U_i}$ is trivial and fix trivialisations $\ke_{U_i}\cong S^dV\otimes\ko_{U_i}$. Then the restriction of (\ref{eqn:relVer1}) to $U_i$ composed
with the induced trivialisation $\PP(\ke_{U_i})\cong \PP(S^dV)\times U_i$ can
be viewed as a flat family $P|_{U_i}\subset U_i\times\PP(S^dV)$ of closed
subschemes of $\PP(S^dV)$ parametrised by $U_i$. %We denote the classifying map for this family by
Let  $\varphi_i\colon U_i\to {\rm Hilb}(\PP(S^dV))$ denote the classifying morphism for this family.
As its image is contained in the $\PGL(S^dV)$-orbit $\kt_0\subset {\rm Hilb}(\PP(S^dV))$, we obtain the Cartesian diagram
$$\xymatrix{P|_{U_i}\ar[d]\ar@{}[dr]|\times\ar[r]& \kp_0\ar[d]\\
U_i\ar[r]&\kt_0.}$$
This proves the assertion.
\end{proof}

The ultimate local version of the proposition is the following function field rephrasing.

\begin{cor}\label{cor:ffieldversion}
Assume $P\to X$ is a Brauer--Severi variety of relative dimension $n-1$ with its period dividing a fixed integer $d$. Then there exists a morphism $\varphi\colon\eta=\Spec(K(X))\to \kt_0$ such that
the generic fibre $P_\eta\to \eta=\Spec(K(X))$  is isomorphic to the pull-back $\varphi^\ast\kp_0$.\qed
\end{cor}

Note that the classifying map $\varphi$ is not unique, not even up to
the action of $\PGL(S^dV)$.

\begin{remark}
The proposition immediately implies upper bounds for the essential dimension of $\PGL(n)$, which are,
however, weaker than the known ones (which are roughly quadratic in $n$), cf.\ \cite[\S\! 10]{MerkEss} or \cite[\S\! 7]{Reichstein}.
\end{remark}

\begin{remark}\label{rem:first} Corollary \ref{cor:ffieldversion} is not entirely new in the literature. Notably, in the special case of the Veronese embedding $\GL(V)/\mu_d\,\hookrightarrow \GL(S^d_iV)$, \cite[Thm.\ 8.1 \& 8.2]{first} (cf.\ also \cite[Thm.\ 8.4]{first}) ultimately also gives the same construction of $Q_i\to B_i$ and proves versality properties with respect to torsors over affine schemes. The universality of the construction though is achieved via the Grassmannian and not the Hilbert scheme as above, and the scope and aims of loc.\ cit.\ and other papers referenced therein are different.
\end{remark}

%%%%%%%%%%%%%%%%%%
\subsection{Classifying maps II}

 Let $X$ be a quasi-projective variety with an ample line bundle $\kl$ and $\pi\colon P\to X$  a Brauer--Severi variety of relative dimension $n-1$. We assume that its period divides a fixed integer $d$, so that we have a relative
 $\ko(d)$ at our disposal and %Then, % tensoring  $\pi_\ast\ko(d)$ with a high power of $\kl$ gives a globally generated sheaf. In other words,
 then $\ke\coloneqq\pi_\ast\ko(d)^\vee$ is locally free of rank $N$. Furthermore, for $m\gg0$ there exists a vector space $W$ and an injection 
\begin{equation}\label{eqn:gg}
\ke\,\hookrightarrow W\otimes \kl^m
\end{equation}
of constant rank.  

\begin{remark} We can be more precise about the dimension of $W$. Once we know
that $\ke^\vee\otimes\kl^m$ is globally generated, then there exist $\rk(\ke)+\dim(X)=N+\dim(X)$ global sections that globally generate $\ke^\vee\otimes\kl^m$, cf.\ \cite[\S\! 5]{HL}. So for this purpose, we may pick $W$ to be of dimension $\dim(W)=N+\dim(X)$. %However, we will later need to pick it a little bigger, cf.\ \S\! \ref{sec:dJS}.
\end{remark}

The induced closed embedding
$$\PP(\ke)\,\hookrightarrow \PP(W\otimes\kl^m)\cong \PP(W)\times X$$
can be composed with (\ref{eqn:relVer1}) to yield a closed embedding over $X$
\begin{equation}\label{eqn:relVer}
\xymatrix@C=10pt@R=15pt{P~\ar[dr]\ar@{^(->}[rr]&&\PP(W)\times X\ar[dl]\\
&X.&}
\end{equation}
Let us now fix an isomorphism $W\cong S^d_iV=S^dV\oplus k^{\oplus i}$ and then view
the smooth family of closed subschemes $P\to X$ of $\PP(S^d_iV)$
as the pull-back  of the universal $\kp_i\to \kt_i$ under the classifying map
$$\psi\colon X\to \kt_i\subset {\rm Hilb}(\PP(S^d_iV)),$$ i.e.\ $P\cong X\times_{\kt_i}\kp_i$.

Altogether we have proved the following global version of Proposition \ref{prop:localUniv}.

\begin{prop}\label{prop:GlobalUniv} Assume $P\to X$ is a Brauer--Severi variety of relative dimension $n-1$ with its period dividing a fixed integer $d$. Then for all $i\geq\dim(X)$ there exists %an extension $S^dV\,\hookrightarrow W$ and 
a morphism $\psi\colon X\to\kt_i$ such that $P$ is isomorphic to the pull-back of
the $i$-th universal Brauer--Severi variety  $\kp_i\to\kt_i$, i.e.\
there exists an isomorphism
$P\cong\psi^\ast\kp_i\cong X\times_{\kt_i}\kp_i$ compatible with the projections to $X$.\qed
\end{prop}

\begin{remark}\label{rem:choices}
The construction of $\psi$ involves a few choices.
 Firstly, we had to choose the vector space $W=S^d_iV$, so essentially its dimension which we tried to keep small depending on the dimension of $X$.
 \smallskip
 
 Secondly, we had to choose a surjection $q\colon S^d_iV^\vee\otimes\ko\twoheadrightarrow \ke^\vee\otimes\kl^m$ and any such surjection can be altered by an isomorphism $S^d_iV^\vee\otimes \ko\congpf S^d_iV^\vee\otimes \ko$. For $X$ integral and pro\-jective, the latter is given by an element in $\PGL(S^d_iV)$ and so the difference between the two classifying maps
is explained by the action of $\PGL(S^d_iV)$ on $\kt_i$. However,
any other choice  of $q$, which usually exists when $X$ is not projective, substantially changes $\psi$. The only way to make $\psi$
canonical (up to the $\PGL(S^dV_i)$-action) would be by choosing $i=h^0(X,\ke^\vee\otimes\kl^m)^\ast-N$, which however makes it dependent on $\ke$.
\end{remark}

\begin{remark} The comparison of the two constructions in 
Propositions \ref{prop:localUniv} \& \ref{prop:GlobalUniv} is a little subtle.
Ideally, one would like  the restriction of $\psi\colon X\to\kt_i$ to one of
the open subsets $U_i\subset X$ in the proof of Proposition  \ref{prop:localUniv} to be linked to the $\varphi_i\colon U_i\to \kt_0$ constructed there, e.g.\ via some closed linear
embedding $\PP(S^dV)\,\hookrightarrow \PP(S^d_iV)$ % given by a linear embedding $\iota\colon S^dV\,\hookrightarrow W$ 
and the induced $\kt_0\,\hookrightarrow \kt_i$.
However, although every trivialisation $\ke_{U_i}\cong S^dV\otimes\ko_{U_i}$ can be extended
to an injection of constant rank $\ke_{U_i}\cong S^dV\otimes\ko_{U_i}\,\hookrightarrow W\otimes\kl^m_{U_i}\cong W\otimes\ko_{U_i}$, the latter might not be induced by a constant $S^dV\,\hookrightarrow W$.

So for $i\geq \dim(X)$ one can arrange things into a diagram that is commutative on the right but not on the left:
$$\xymatrix{\Spec(K(X))\ar[d]^-\varphi&\cdots&&X\ar[d]^{\psi}\ar@{=}[r]&X\ar[d]^\psi&\cdots\\
\kt_0&\cdots&\kt_i\ar@{^(->}[r]&\cdots\kt_{\dim(X)}\ar@{^(->}[r]&\kt_{\dim(X)+1}&\cdots}$$
\end{remark}

\subsection{Compatibilities} For certain problems (not for the period-index conjecture), it sometimes
suffices to consider the case $d=n$ and
it is possible to link any other choice to it in terms of  classifying maps. We will restrict the discussion
to the case of the smallest universal Brauer--Severi varieties $\kt_0$ and wish to compare
them for the two choices $(d,n)$ and $(n,n)$. So we are considering the two universal Brauer--Severi varieties $$\pi^d\colon \kp_0(d,n)\to \kt_0(d,n) ~\text{ and }~\pi^n\colon\kp_0(n,n)\to \kt_0(n,n).$$

\begin{prop}\label{prop:compa}
There exists a morphism $\varphi\colon\kt_0(d,n)\to \kt_0(n,n)$ such that $\kp_0(d,n)\cong\varphi^\ast\kp_0(n,n)$.
\end{prop} 

\begin{proof} Taking the $n/d$-th symmetric power of  the natural trivialisation
$S^dV^\vee\otimes\ko\cong \pi^d_\ast\ko_{\pi^d}(d)$ over $B_0(d,n)$ induces
an isomorphism $S^{n/d}(S^dV^\vee)\otimes\ko\cong S^{n/d}\pi^d_\ast\ko_{\pi^d}(d)$.
It commutes with the two quotient maps $S^{n/d}(S^dV^\vee)\twoheadrightarrow S^nV^\vee$
and $S^{n/d}\pi^d_\ast\ko_{\pi^d}(d)\twoheadrightarrow \pi^d_\ast\ko_{\pi^d}(n)$ and, therefore,
defines a trivialisation $S^nV^\vee\otimes\ko\cong \pi_\ast^d\ko_{\pi^d}(n)$. Apply the construction
of the classifying morphism $\varphi\colon B_0(d,n)\to B_0(n,n)$ in \S\! \ref{sec:ClassmapsI} to conclude.
\end{proof}

%\begin{remark}
%In addition, $\kt_0(d,n)$ can be reconstructed from the Brauer--Severi variety $\kp=\kp_0(n,n)\to \kt_0(n,n)$ in the following week sense. The relative symmetric product of $\kp$ is a Brauer--Severi variety $\pi\colon S^d\kp\to\kt_0(n,n)$ (of relative dimension $N-1$) for which $\pi^\ast[\kp]\in \Br(S^d\kp)$ has order dividing $d$. Hence, the pulled back Brauer--Severi variety
%$S^d\kp\times_{\kt_0(n,n)}\kp\to S^d\kp$ comes with a relative $\ko(d)$. The classifying  (rational) map
%$ S^d\kp\to \kt_0(d,n)$ is dominant, for $S^d\kp_0(d,n)\to \kt_0(d,n)$ comes with a rational section
%$$\xymatrix@R=15pt{&\kt_0(d,n)\ar[d]\\
%S^d\kp_0(n,n)\ar@{-->>}[ur]\ar[r]&\kt_0(n,n)}$$
%\end{remark} 

\begin{remark}
The morphism $\varphi\colon B_0(d,n)\to B_0(n,n)$ induces pull-back maps
$$\varphi^\ast\colon \Pic(B_0(n,n))\to\Pic(B_0(d,n))~\text{ and }~\varphi^\ast\colon \Br(B_0(n,n))\to\Br(B_0(d,n)).$$
In the next section, all these groups will be shown to be cyclic of finite order.
For example, $\Br(B_0(n,n))\cong\ZZ/n\ZZ$ and $\Br(B_0(d,n))\cong\ZZ/d\ZZ$,
both generated by the Brauer class of the respective universal Brauer--Severi varieties. 
Hence, by construction of $\varphi$, the two restriction maps $\varphi^\ast$ are the natural quotients.
\end{remark}

With similar techniques one can prove, for example, the existence of a classifying morphism
$$\varphi\colon B_0(d,n)\times B_0(\ell,\ell)\to B_0(\ell n,\ell d),$$
such that the pull-back $\varphi^\ast Q_0(\ell n,\ell d)$ of the universal Brauer--Severi variety
over $B_0(\ell n,\ell d)$ is isomorphic to the exterior tensor product $Q_0(d,n)\boxtimes Q_0(\ell,\ell)$
of the universal Brauer--Severi varieties over the two factors.
%%%%%%%%%%%%%%%%%%%%%
%%%%%%%%%%%%%%%%%%%%%%%%%%%%
\section{Brauer, Picard, and fundamental group}\label{sec:PIBrPic}
In this section we compute the three groups in the title for the universal Brauer--Severi variety $\kp_i\to\kt_i$. The computation will reveal a fundamental difference between the case $i=0$ and $i>0$.\smallskip

As in the previous section, we let $V$ be a $k$-vector space of dimension $n$ and $d$ an integer
dividing $n$. However, from now on we will assume that $k$ is algebraically closed.
For other fields $k$, the description of the fundamental group and the Brauer
group (but not of the Picard group) would need to be modified, cf.\ Remark \ref{rem:Serre}.\smallskip

 Recall that by construction, the universal Brauer--Severi variety $\kp_0\to\kt_0$ parametrises 
all closed subschemes of $\PP(S^dV)$ that are projective equivalent to the image
of a Veronese embedding. Similarly, the global universal Brauer--Severi variety
$\kp_i\to\kt_i$, $i>0$, parametrises all subschemes contained in the bigger projective
space $\PP(S^d_iV)$ projective equivalent to the image of a Veronese embedding. \smallskip

We begin by collecting a few standard facts for classical linear algebraic groups over an algebraically closed field $k$:
\begin{enumerate}
\item The fundamental group of $\PGL(\ell)$ is cyclic of order $\ell$, i.e.\
$\pi_1(\PGL(\ell))\cong\mu_\ell$. The universal cover is given by the $\mu_\ell$-quotient $\SL(\ell)\twoheadrightarrow \PGL(\ell)$.
\item The Picard group of $\PGL(\ell)$ is also cyclic of order $\ell$, i.e.\
$\Pic(\PGL(\ell))\cong\ZZ/\ell\ZZ$, cf.\ \cite[Cor.\ 4.6]{FossumeIversen}.
The \'etale cyclic cover associated with its torsion generator
is again  the quotient $\SL(\ell)\twoheadrightarrow \PGL(\ell)$.
\item The Brauer group $\Br(\PGL(\ell))$ is trivial, cf.\ \cite[Cor.\ 4.3]{Iversen}.
\item The general linear group $\GL(\ell)$ has trivial Picard and Brauer group.
The former follows from realising $\GL(\ell)$ as an open subset of the affine space 
of dimension $\ell^2$. The latter is well-known and can be found in \cite[Prop.\ 7.4]{RS}. By means of the determinant map, the fundamental group of $\GL(\ell)$ is identified with $\pi_1(\GL(\ell))\cong\pi_1(\GL(1))\cong\ZZ$.\end{enumerate}

%%%%%%%%%%%%%
\subsection{Fundamental group}
The easiest of the three computations is the one of the fundamental group.
The difference between the two cases $B_0$ and $B_i$, $i>0$, in the following proposition
ultimately reflects the discussion in the previous section:
We expect in both cases the compactification of $\kt_i$ to be simply connected. Since for $i\gg0$ the boundary of the GIT-compactification $\kt_i\subset\bar\kt_i$ is of codimension at least two, also $\kt_i$ should be simply connected.
On the other hand, loops around the codimension one components of the boundary 
of any projective compactification $ \kt_0\subset\bar\kt_0$ of the affine $\kt_0$ 
potentially produce non-trivial elements in the fundamental group.

\begin{prop}\label{prop:Funda}
The space $\kt_i$ is simply connected if and only if $i>0$. Furthermore,
the fundamental group of $\kt_0$ is cyclic of order $dN/n$:
$$\pi_1(\kp_i)\cong\pi_1(\kt_i)\cong\begin{cases}
    \{1\},& \text{if } i>0\\
    \ZZ/(dN/n)\ZZ,              & \text{if } i=0.
\end{cases}$$
\end{prop}

Recall that we work over an algebraically closed ground field $k$. For $k=\CC$,
one can alternatively work with the \'etale or the topological fundamental group.

\begin{proof}
We use the representation of $\kt_i$ as the quotient of $\GL(S^d_iV)$ by the action of the group 
$$\left(\begin{matrix}\GL(V) &\ast\\
0&\GL(k^{\oplus i})
\end{matrix}\right)$$
and the fact that  $\det\colon\GL\to \GG_m$ induces an isomorphism
$\pi_1(\GL)\cong\pi_1(\GG_m)\cong\ZZ$. Now, for $i\ne0$ the inclusion $\GL(k^{\oplus i})\,\hookrightarrow\GL(S^d_iV)$ is compatible with taking determinants. Hence, the induced map $\pi_1(\GL(k^{\oplus i}))\to \pi_1(\GL(S^d_iV))$ in the standard exact sequence for fibrations is surjective, which implies the first assertion.\smallskip

For $i=0$ it suffices to prove that  $\ZZ\cong\pi_1(\GL(V))\to\pi_1(\GL(S^dV))\cong\ZZ$ induced by $\GL(V)\to\GL(S^dV)$ is given by multiplication with $dN/n$. This follows from the observation that the determinant on
$\GL(S^dV)$ restricted to $\GL(V)$ is the $(dN/n)$-th power of the determinant on $\GL(V)$.
Alternatively, one can use the exact  sequence  of homotopy groups for the quotient $\kt_0=\PGL(S^dV)\,/\,\PGL(V)$ 
of the subgroup(!) $\PGL(V)\subset \PGL(S^dV)$.\smallskip

The isomorphism $\pi_1(\kp_i)\cong\pi_1(\kt_i)$ follows from the long exact homotopy sequence for
the fibration $\kp_i\to\kt_i$ which has simply connected fibres.
\end{proof}

%%%%%%%%%%%%%
\subsection{Picard group} 
Any classifying map $\psi\colon X\to \kt_i$, cf.\ Proposition \ref{prop:GlobalUniv}, 
pulls back line bundles on $\kt_i$ to line bundles on  $X$. So one should maybe 
not expect $\kt_i$ to carry too many non-trivial line bundles. Note, however, that the construction
of the classifying map $\varphi$ already involved an ample line bundle on $X$.\smallskip

As for the fundamental group, there is a difference in the Picard groups in the two cases
$i=0$ and $i>0$.  Since $\kt_i$ for $i>0$ admits a compactification with a boundary
of codimension at least two, its Picard group cannot be trivial and not even torsion.

\begin{prop}\label{prop:Picard}
The Picard group of $\kt_i$ is described by:
 $$\Pic(\kt_i)\cong\begin{cases}
    \ZZ,& \text{if } i>0\\
    \ZZ/(dN/n)\ZZ,              & \text{if } i=0.
\end{cases}$$
 Accordingly,
 $$\Pic(\kp_i)\cong
\begin{cases}
    \ZZ^{\oplus 2},& \text{if } i>0\\
    \ZZ/(dN/n)\ZZ\oplus \ZZ,              & \text{if } i=0.
\end{cases}$$ 
\end{prop}

\begin{proof}
The second assertion follows from the first, as the Picard group of
any Brauer--Severi variety $P\to X$ is described by $$\Pic(P)\cong\Pic(X)\oplus\ZZ,$$
where the second summand is generated by the relative $\ko(d)$ with $d$ the period of
$[P]\in\Br(X)$.\smallskip

The first assertion follows from \cite[Thm.\ 4]{Popov}: For a connected linear algebraic group $G$ with
$\Pic(G)=0$, the Picard group of the quotient by a closed subgroup $H\subset G$ is
described by the short exact sequence
\begin{equation}\label{eqn:sesChar}
\xymatrix{0\ar[r]&\chi_G(H)\ar[r]&\chi(H)\ar[r]&\Pic(G\,/\,H)\ar[r]&0.}
\end{equation}
Here, $\chi(H)$ is the character group of $H$ and $\chi_G(H)=\im(\chi(G)\to\chi(H))$.
 \smallskip

 In the case $i=0$,  we set $G=\GL(S^dV)$ and $H=\GL(V)\,/\,\mu_d\subset\GL(S^dV)$. Then,
 since $d\mid n$, the group
 $\chi(H)$ is freely generated by $\det_V$ and the group $\chi_G(H)$ is generated by the restriction of $\det_{S^dV}$ on $\GL(S^dV)$ to $\GL(V)$, which is, as observed earlier, the $(dN/n)$-th power of the determinant $\det_V$.  %$\chi_{\GL(S^dV)}(\GL(V))=(N/n)\ZZ\,\hookrightarrow \chi(\GL(V))=\ZZ$,which proves $\Pic(\kt)\cong \ZZ/(N/n)\ZZ$.
Hence, $$\Pic(\kt_0)\cong\coker\left( \ZZ\cdot{\det}_V^{dN/n}\to \ZZ\cdot{\det}_V\right)\cong\ZZ/(dN/n)\ZZ.$$

For $i>0$, we set $G=\GL(S^d_iV)$ and $$H=\left(\begin{matrix}\GL(V)\,/\,\mu_d&\ast\\
0&\GL(k^{\oplus i})\end{matrix}\right).$$ Then $\Pic(\kt_i)$ is the cokernel of the restriction
map
$$\chi(\GL(S^d_iV))\cong\ZZ\cdot{\det}_{S^d_iV}\to \chi(H)=\ZZ\cdot{\det}_V\oplus\ZZ\cdot{\det}_{k^{\oplus i}},$$
where the generator $\det_{S^d_iV}$ maps to $({\det}_V^{dN/n},\det_{k^{\oplus i}})$. This proves the claim.
\end{proof}
The proposition assumes that $d$ divides $n$ and this is the only case of interest to us. Without this assumption
similar arguments as above prove that $\Pic(\kt_0)$ is cyclic of order $\gcd(d,n)N/n$ and that for $i>0$ the Picard group
$\Pic(\kt_i)$ is still cyclic of infinite order (with a different generator).

\begin{remark}\label{rem:unicover}
According to Propositions \ref{prop:Funda} and \ref{prop:Picard}, we have $$\pi_1(\kt_0)\cong\ZZ/(dN/n)\ZZ~\text{ and }~\Pic(\kt_0)\cong\ZZ/(dN/n)\ZZ.$$
Hence, the cyclic \'etale cover associated with the generator $L\in\Pic(\kt_0)$ describes the universal
cover of $\kt_0$:
%$\kt$ and its universal cover $\widetilde\kt$:
$$\xymatrix{\widetilde\kt_0\cong {\bf Spec}\left(\bigoplus_{i=0}^{(dN/n)-1} L^i\right)\ar@{->>}[rr]^-{(dN/n):1}&&\kt_0}$$
%$$\xymatrix{\widetilde\kt\ar[dr]\ar[r]^-{d:1}&{\bf Spec}\left(\bigoplus_{i=0}^{(N/n)-1} L^i\right)\ar[d]^-{(N/n):1}\\&\kt.}$$
\end{remark}

%%%%%%%
\subsection{Brauer group}
%We consider the Brauer--Severi variety $\kp\to \kt$ or, more generally, $\kp_W\to\kt_W$.
It is known that the kernel of the pull-back map
to the Brauer--Severi variety is spanned by its class $\pmb{\alpha}_i\coloneqq
[\kp_i]$, i.e.\ 
$$\ker\left(\Br(\kt_i)\to\Br(\kp_i)\right)=\langle\pmb{\alpha}_i\rangle.$$

In fact, $[\kp_i]$ not only generates the kernel but the entire Brauer group of
$\kt_i$, as is shown by the next proposition, which resembles a result of Saltman \cite[Thm.\ 3.3]{SaltmanBrauer}.

\begin{prop}\label{prop:BrTW}
For all $i\geq0$, the class $\pmb{\alpha}_i$ of the global universal Brauer--Severi variety  $\kp_i\to\kt_i$ gene\-rates $\Br(\kt_i)$ and the Brauer group of $\kp_i$ is trivial, i.e.\ 
$$\Br(\kt_i)=\langle\pmb{\alpha}_i\rangle~\text{ and }~\Br(\kp_i)=\{0\}.$$
\end{prop}

In the next section, see Proposition \ref{prop:piUniv}, we will explain that this implies $$\Br(\kt_i)\cong\ZZ/d\ZZ$$
if $d\mid n\mid d^e$ for some $e$.

\begin{proof} We use the standard fact that
pull-back induces a short exact sequence
$$0\to\langle{\pmb\alpha}_i\rangle\to\Br(\kt_i)\to\Br(\kp_i)\to0,$$
cf.\ \cite[Part II, Thm.\ 2]{Gabber}. Thus, it suffices to show that $\Br(\kp_i)$ is trivial. \smallskip

For notational simplicity, we will restrict to the case $\kt_0$. Since
$\Br(\PGL(S^dV))=0$, it suffices to show that the projection $\PGL(S^dV)\twoheadrightarrow\kp_0$ 
admits a rational section. Fix a basis $v_1,\ldots,v_n\in V$ and consider the induced
basis $v_I\in S^dV$ with $|I|=d$. Among them there are $v_{I_1}=v_1^d,~v_{I_2}=v_1^{d-1}v_2,\ldots,v_{I_n}=v_1^{d-1}v_n$. Then for $g\in \GL(S^dV)$ we define a matrix $\left(a_{ij}\right)$, $i=2,\ldots,n$ and $j=1,\ldots,n$,
by $$g(v_{I_i})=\sum_{j=1}^na_{ij}v_{I_j}+\sum_{I\ne I_k}a_Iv_I.$$
A direct computation shows that the transversal intersection of the stabiliser $\PGL(V)_x\subset\PGL(S^dV)$ and the linear subspace cut out by the $n^2-n$ linear equations expressed by $$\left(a_{ij}\right)=
\left(\begin{matrix}0&1&0&\ldots&0\\
                              0&0&1&\ldots &0\\
                              0&0&0&\ldots &1
\end{matrix}\right)$$
consists of the unit matrix. As $n^2-n=\dim\PGL(V)_x$, this linear subspace will then also intersect the generic
fibre of $\PGL(S^dV)\to\kp_0$  transversally in one point, i.e.\ it defines a rational section.\smallskip

As an aside, the projection $\PGL(S^dV)\to\kt_0$ does not admit any rational section, for $\Br(\kt_0)$ is definitely not trivial. In particular, there is no linear subspace as above of codimension $n^2-1$ that
would intersect $\PGL(V)\subset\PGL(S^dV)$ transversally in the unit matrix.
\end{proof}

\begin{remark}
For $i\geq 1$, there exists a compactification $\kt_i\subset\bar\kt_i$
with a boundary $\partial\kt_i$ of codimension $\geq2$, see Example \ref{ex:codim2}. 
We expect $\bar\kt_i$ to be singular along $\partial\kt_i$ and, in principle, the GIT construction can be used to describe the singularities. If $\bar\kt_i$ were smooth,
% and we expect it to be badly singular. Suppose it is smooth, which seems unlikely, 
then purity allows one to extend ${\pmb\alpha_i}$, i.e.\ restriction defines
an isomorphism $$\Br(\bar\kt_i)\cong\Br(\kt_i)\cong\ZZ/d\ZZ.$$
If $k=\CC$, this could be translated into $H^3(\bar\kt_i,\ZZ)_{\rm tors}\cong\ZZ/d\ZZ$, using
Proposition \ref{prop:piUniv} below, and, in particular, $\kt_i$ would not be rational.\smallskip

In the affine case $\kt_0$, the Brauer class ${\pmb \alpha}_0$ is ramified along certain
components of the boundary $\partial\kt_0\subset\bar\kt_0$. In other words, $\Br(\bar\kt_0)\to\Br(\kt_0)$ is not surjective. Indeed, otherwise there would be no ramified Brauer classes on any $X$. In other words,
the unramified Brauer group of $K(\kt_0)$ is trivial.
In the algebraic setting, this resembles a result of Saltman \cite[Thm.\ 2.9]{SaltmanBrauer}
or \cite[Cor.\ 14.30]{SaltmanLectures}.
\end{remark}

\begin{remark}\label{rem:Serre}
For arithmetic applications it could be of interest to know the Brauer group of $\kt_i$ over
a non-closed field $k$. In fact, assuming $i>0$ and ${\rm char}(k)=0$, the Hochschild--Serre spectral sequence
gives an exact sequence $\Br(k)\to\Br_1(\kt_i)\to H^1(k,\Pic(\kt_{i}\times\bar k))$, where
$\Br_1(\kt_i)=\ker(\Br(\kt_i)\to \Br(\kt_{i}\times\bar k))$. % and $\bar\kt_i$ is the base change of $\kt_i$ to an algebraic closure of $k$. 
We use the assumption $i>0$  to ensure $H^0(\kt_i,\GG_m)=k^\ast$
and  $\Pic(\kt_i\times\bar k)\cong \ZZ$. Furthermore, the Galois action on the latter is trivial
and since $\kt_i(k)\ne\emptyset$, the Brauer group injects into $\Br(\kt_i)$. Altogether, this proves
$$\Br(k)\cong\Br(\kt_i)$$
for $i>0$ and an arbitrary field $k$ of characteristic zero. It seems likely that  the arguments can be pushed further to prove the result in full generality, but we will not pursue this here.
\end{remark}

%%%%%%%%%%%%%%%%
\subsection{Period and index of the  generator}\label{sec:periodindexUniv}
After having verified that the Brauer group of $\kt_i$ is cyclic, 
we now determine the period and the index of its generator.

\begin{prop}\label{prop:piUniv}
Assume that $d$ and $n$ share the same prime factors and $d\mid n$. Then the period and the index of the universal Brauer class $\pmb{\alpha}_i$ are
given by
$$\per(\pmb{\alpha}_i)=d~\text{ and }~\ind(\pmb{\alpha}_i)=n.$$
In particular, $$\Br(\kt_i)\cong\ZZ/d\ZZ.$$
\end{prop}

\begin{proof} Let us first show that it suffices to prove the assertion for one $i$.\smallskip

Assume it holds for $i=0$. Then use $$d=\per(\pmb{\alpha}_0)=\per(\pmb{\alpha}_i|_{\kt_0})\mid\per(\pmb{\alpha}_i)~\text{ and }~n=\ind(\pmb{\alpha}_0)=\ind(\pmb{\alpha}_i|_{\kt_0})\mid\ind(\pmb{\alpha}_i),$$
for  $\kp_i|_{\kt_0}\cong\kp_0$. As by construction, $\kp_i\to\kt_i$ comes with a relative
$\ko(d)$ and its relative dimension is $n-1$, we also  have
$\per(\pmb{\alpha}_i)\mid d$ and $\ind(\pmb{\alpha}_i)\mid n$, which proves the assertion for $i$. \smallskip

Conversely, assume the assertion holds for some $i>0$. Then use the rational classifying
map $\varphi\colon B_i\to B_0$, see Corollary \ref{cor:ffieldversion} or (\ref{eqn:proj}) in \S\! \ref{sec:Fibring}, for which  $\varphi^\ast\kp_0\cong\kp_i$ on
some open subset of $B_i$. This implies $$d=\per({\pmb \alpha}_i)\mid\per({\pmb \alpha}_0)\mid d~\text{ and 
}~n=\ind({\pmb \alpha}_i)\mid\ind({\pmb \alpha}_0)\mid n,$$ which is enough to conclude assertion for $i=0$. Clearly, the equivalence of the cases $i=0$ and $i>0$ holds separately for
the two assertions concerning the period and the index. \smallskip

Let us now prove the assertion for some $i\gg0$. We first verify that the period of $\pi\colon\kp_i\to\kt_i$ is
exactly $d$. According to Lemma \ref{lem:Or}, if its period is, say, $d'$ then there exists
a relative $\ko(d')$ which can of course be modified by any invertible sheaf in $\kt_i$.
In \S\! \ref{sec:covfam}, we will view $\kp_i$ as a closed subscheme of $ \kt_i\times\PP(S^d_iV)$ and
the fibres of the second projection $p\colon\kp_i\to\PP(S^d_iV)$
as closed subschemes $\kt_{i,x}\subset \kt_i$.
We have proved that $\Pic(\kt_i)\cong\ZZ$, see Proposition \ref{prop:Picard}, and we will later prove
that the restriction map $\Pic(\kt_i)\congpf\Pic(\kt_{i,x})$ is an isomorphism, see Lemma \ref{lem:PicBiBix}.
It is here where we use $d\mid n$.\smallskip

 Therefore, after twisting by line bundles coming from $\kt_i$, we
may assume that $\ko(d')$ restricts trivially to the fibres $\kt_{i,x}\subset \kp_i$ of $\kp_i\to\PP(S^d_iV)$. On the other hand, $\kt_{i,x}$ admits a normal projective compactification with a boundary of
codimension at least two, cf.\ Remark \ref{rem:codimBound}. Hence, $p_\ast\ko(d')$ is an invertible sheaf on $\PP(S^d_iV)$ and so of the
form $\ko(\ell)$. Hence, $\ko(d')\cong p^\ast\ko(\ell)$. Since $\ko(1)$ under the Veronese
embedding $\PP(V)\,\hookrightarrow\PP(S^dV)\,\hookrightarrow\PP(S^d_iV)$ pulls back to
$\ko(d)$, we find $d'=\ell \cdot d$. This proves  $\per({\pmb\alpha}_i)=d$, because $\ko(d)$ exists
and, hence, $d'\mid d$.\smallskip

Let us now turn to the computation of the index. Here we  use that for any prime $p$ and a positive integer $\ell$ there exists a Brauer class $\alpha$ of period $p$ and index $p^\ell$ on some quasi-projective variety $X$. For a ramified example see e.g.\ \cite[\S\! 1]{dJ} (the argument there generalises from $p=2$ to arbitrary prime $p$) and for examples of unramified classes, at least in characteristic zero, see Gabber's appendix to \cite{CT3}. 
Shrinking to an open subset of $X$ we may assume that $\alpha$ is represented by a Brauer--Severi variety $P\to X$, which is then isomorphic to the pull-back  of $\kp_0\to\kt_0=\kt_0(d,n)$,
with $n\coloneqq p^\ell$ and $d\coloneqq p$, under a classifying
morphism $\psi\colon X\to \kt_0$. Using the homogeneous nature of $\kp_0\to\kt_0$, this 
immediately proves that $\kp_0\to \kt_0(p^\ell,p)$ has index $p^\ell$. More precisely,
if the index of $\kp_0\to\kt_0$ were smaller than $p^\ell$, then its class could be represented by a Brauer--Severi variety of relative dimension $<p^\ell-1$ over
some non-empty open subset $U\subset\kt_0$  using the transitive action, we may assume that $U$ contains the image of the generic point of $X$, which would show that the class of $P\to X$ is of index $<p^\ell$.\smallskip

We now claim that then also $\kp_0\to \kt_0(p^\ell,p^m)$ for any $m>0$ has index $p^\ell$.
To this end, consider the rational classifying  map $\kt_0(p^\ell,p)\to \kt_0(p^\ell,p^m)$ associated
with $\kp_0\to\kt_0(p^\ell,p)$ but considered with the relative
$\ko(p^m)$ instead of $\ko(p)$, cf.\ Proposition \ref{prop:compa}. As in the argument before, since the index can only drop (again
the homogeneity is used), this is enough to deduce the claim.\smallskip

For the general case, let  $n=\prod p_i^{\ell_i}$ and $d=\prod p_i^{m_i}$ be the prime factorisations of the given $n$ and $d$, which are assumed to have the same prime factors. In order to prove that the index
of $\kp_0\to \kt_0(d,n)$ has index $n$, it suffices to show that its pull-back
under the classifying morphism $\prod\kt_0(p_i^{\ell_i},p_i^{m_i})\to \kt_0(d,n)$
associated with the exterior product $$\bigotimes \pi_i^\ast \kp_0(p_i^{\ell_i},p_i^{m_i})\to \prod\kt_0(p_i^{\ell_i},p_i^{m_i})$$ has index $n$. As we are working over an algebraically closed
field and so all varieties admit rational points, the individual pull-backs $\pi_i^\ast \kp_0(p_i^{\ell_i},p_i^{m_i})$ remain of period $p_i^{m_i}$ and
index $p_i^{\ell_i}$. Hence, according to a well-known fact, cf.\ \cite[Prop.\ 2.8.13]{GS}
or the general result in \cite[Thm.\ 3]{AW3},  the index of their product is then indeed $n=\prod p_i^{\ell_i}$.
\end{proof}

Ideally, we would have liked to give again a global proof for the assertion about the index, using e.g.\ the description of the cohomology of $\kt_i$ in  \S\! \ref{sec:Coh}. In fact, we expect the Brauer--Severi variety $\kp_i\to\kt_i$ to be topologically of period $d$ and index $n$, which should be reflected by the integral cohomology of $\kp_i$. Unfortunately, the integral
cohomology of $\PGL(N,\CC)$ and its classifying space is quite intricate, cf.\ \cite{AW2, Gu}.
Furthermore, the assumption that $n$ and $d$ share the same prime factors  would need to be used here and it is not obvious how this could come in. Note that up to this point in the proof above, only $d\mid n$ was used.\smallskip

The proposition has the following obvious consequence.

\begin{cor} Assume $k$ is a given field and $d$ and $n$ are two fixed positive integers with the same prime factors and such that $d\mid n$. Then there
exists a quasi-projective variety $X$ over $k$ with a Brauer class $\alpha\in\Br(X)$ of period
$\per(\alpha)=d$ and index $\ind(\alpha)=n$.\qed
\end{cor}

As far as we are aware, this result is not stated explicitly as such anywhere in the literature, but it is
certainly well-known to the experts, cf.\ \cite[Rem.\ 4.7.6]{GS}. In fact, one can turn the above argument around and give an alternative argument for the key step
in the above proof: First construct $P\to \Spec(K(X))$ of dimension $n-1$ over the function field of some quasi-projective variety $X$ with given $\per([P])=d$ and $\ind([P])=n$. Then apply the universality property, see Proposition \ref{prop:localUniv} and
Corollary \ref{cor:ffieldversion}, which shows that $P$ is pulled back from $\kp_0\to\kt_0$ via
the classifying map $\Spec(K(X))\to \kt_0$. Hence, period and index of $[\kp_0]\in\Br(\kt_0)$ cannot
be smaller than $d$ and $n$, respectively.

\begin{remark}
Note that the period-index conjecture would predict that
$\ind(\pmb{\alpha}_0)\mid\per(\pmb{\alpha}_0)^{\dim\kt_0-1}$, i.e.
$n\mid d^{\dim\kt_0-1}$. As $\dim(\kt_0)=N^2-n^2$ is rather large,
this is indeed the case.
\end{remark}

\begin{remark}
For $i=0$, we can consider the universal cover $\widetilde \kt_0\to\kt_0$, which is cyclic 
of degree $\ell\coloneqq dN/n$, cf.\ Remark \ref{rem:unicover}, and it is then natural to ask about the pull-back map
$$\ZZ/d\ZZ\cong\langle{\pmb\alpha}_0\rangle\cong\Br(\kt_0)\to\Br(\widetilde\kt_0).$$
We predict that typically this maps is in fact injective. Hence, ${\pmb \alpha}_0$ and no non-trivial power
of it will be trivialised by $\widetilde\kt_0$. Indeed, otherwise one could conclude the same for
arbitrary (ramified) Brauer classes over all function fields which would prove a far too strong cyclicity result. In fact, for e.g.\ $d=n=3$, the degree of the covering $\widetilde \kt_0\to\kt_0$ is not
divisible by the period of $\pmb\alpha_0$ and so cannot split it.
\end{remark}

\begin{cor}
Assume $k=\CC$. Then $H^2_{\text{\rm \'et}}(\kt_0,\QQ/\ZZ)\cong\Br(\kt_0)\cong\ZZ/d\ZZ$
and for $i>0$, there is an extension $$0\to \QQ/\ZZ\to H^2_{\text{\rm\'et}}(\kt_i,\QQ/\ZZ)\to\Br(\kt_i)\cong\ZZ/d\ZZ\to0.$$
\end{cor}
For further information on the rational cohomology of $\kt_i$ see \S\! \ref{sec:Coh}.
\begin{proof} The proof follows a standard argument, cf.\ \cite[Lem.\ 9.3]{CTSa}. The special Brauer group $\SBr(\kt_i)\cong H^2_{\text{\rm\'et}}(\kt_i,\QQ/\ZZ)$, see \cite{BrauerII,HuyMa}, sits
in a short exact sequence 
$$0\to \Pic(\kt_i)\otimes \QQ/\ZZ\to H^2_{\text{\rm\'et}}(\kt_i,\QQ/\ZZ)\to\Br(\kt_i)\cong\ZZ/d\ZZ\to0,$$
where we refer to Proposition \ref{prop:piUniv} for the isomorphism. This proves both
assertions.
\end{proof}
%%%%%%%%%%%%%%%%
\subsection{Universal division algebra}
As the Brauer--Severi variety $\kp_i\to \kt_i$ is of relative dimension $n-1$ and of index $n$,
see Proposition \ref{prop:piUniv}, 
the associated Azumaya sheaf $\ka_i$ is a division algebra at the generic point, i.e.\
$D_i\coloneqq \ka_i\otimes K(\kt_i)$ is a central division algebra over the function field $K(\kt_i)$.\smallskip

To simplify the notation, let us restrict to the case $i=0$. Then $\ka=\ka_0$ is the locally free sheaf
of algebras that can be described as the homogeneous bundle $$\PGL(S^dV)\times_{\PGL(V)}(V\otimes V^\vee)$$
on $B_0=\PGL(S^dV)/\PGL(V)$. Since $\kt_0$ is affine, $\ka$ is the sheaf associated to the
$\ko(\kt_0)$-algebra $A=H^0(B_0,\ka)$, which can be understood as the space of
$\PGL(V)$-equivariant morphisms $\PGL(S^dV)\to V\otimes V^\vee$.\smallskip

Note that by the relative Euler sequence for Brauer--Severi varieties, there exists
a natural projection $A\to H^0(Q_0,{\mathcal T}_\pi)$, where ${\mathcal T}_\pi$ is the relative
tangent sheaf of the projection $\pi\colon \kp_0\to \kt_0$. For  generic $s\in A$ and its
image $\bar s\in H^0(\kp_0,{\mathcal T}_\pi)$, the zero set $Z(\bar S)\subset Q_0$ defines
a generically finite morphism $Z(\bar s)\to \kt_0$ of degree $n$ that splits $\pmb \alpha_0$.\smallskip

Passing to the generic points defines a splitting extension $L/K(\kt_0)$ for $\pmb\alpha_0\in \Br(K(B_0))$.
It would be interesting and should in principle be possible to compute its Galois group $G(d,n)\coloneqq{\rm Gal}(L/K(\kt_0))$. The Galois group of any other Brauer--Severi variety $P\to X$ of index $n$ and period $d$ is then isomorphic to a subgroup of $G(d,n)$.

%%%%%%%%%%%%%%%
\section{The geometry of the universal Brauer--Severi variety}\label{sec:contr}
This section discusses some basic features of the geometry of the universal Brauer--Severi varieties
$B_i$ beyond their Picard, Brauer and fundamental group. We will first consider the natural fibrations
over the Grassmannians, which will lead us to coverings by subvarieties with trivial Brauer groups and eventually give a complete description of their rational cohomology.

\subsection{Fibring $\kt_i$ by $\kt_0$}\label{sec:Fibring}
%Assume $Z\subset\PP(W)$ is a closed subvariety that is projective equivalent to the imageof the composition$\xymatrix{\nu_{d,\iota}\colon\PP(V)\ar@{^(->}[r]&\PP(S^dV)\ar@{^(->}[r]&\PP(W),}$see \S\! \ref{sec:TW}. In other words, it defines a point $[Z]\in\kt_W\subset {\rm Hilb}(\PP(W))$. In particular, the span of $Z$ is a linear subspace$\langle Z\rangle\subset\PP(W)$ of dimension $N-1$ and as such defines a point in the Grassmannian${\rm Gr}(N,W)$. 
In \S\! \ref{sec:comp} we already encountered the morphism
\begin{equation}\label{eqn:quotp}
p_i\colon \kt_i\to {\rm Gr}_i\coloneqq{\rm Gr}(N,S^d_iV),~[Z]\mapsto \langle Z\rangle,
\end{equation}
which there occurred naturally as a morphism of quotients
$$p_i\colon\kt_i\cong U_N/\GL(V)\twoheadrightarrow U_N/\GL(S^dV)\cong{\rm Gr}_i.$$
The projection $p_i$ can alternatively be described as the classifying morphism
for the quotient bundle $ S^d_iV^\vee\otimes\ko_{\kt_i}\twoheadrightarrow\pi_{i\ast}\ko_{\pi_i}(d)$,
where $\pi_i\colon \kp_i\to \kt_i$ is the projection. It is not difficult to see that $p_i\colon\kt_i\to{\rm Gr}_i$ is a Zariski locally trivial fibration with fibres isomorphic to $\kt_0\cong \GL(S^dV)\,/\,\GL(V)$, see also Remark \ref{rem:contr} below. In fact, if $\GL(S^dV)\subset \GL(S^d_iV)$
is the natural embedding defined by the fixed splitting $S^dV\oplus k^{\oplus i}=S^d_iV$, then
the natural action of $\GL(S^d_iV)$ on $\kt_i$ restricts to an action of
$\GL(S^dV)$, the quotient of which is again
$p_i\colon \kt_i=\GL(S^d_iV)\,/\,\GL(V,i)\twoheadrightarrow{\rm Gr}(N,S^d_iV)\cong\kt_i\,/\,\GL(S^dV)$.\smallskip

The fibration sheds a new light on the computations  of the various groups, $\Br$, $\Pic$, and $\pi_1$, 
which we presented in each case for $\kt_0$ and $\kt_i$, $i>0$ separately.\smallskip

$\bullet$ The long exact sequence of homotopy groups reads:\TBC{Need reference for Grassmannian. Maybe Bott--Tu, Hatcher, Milnor... Follows from long homotopy sequence for
${\rm Gr}_i(N,W)=U(n)/U(N)\times U()$. Is there an algebraic version, positive char.}
$$\xymatrix@R=0.1pt{\cdots\ar[r]&\pi_2({\rm Gr}_i)\ar[r]&\pi_1(\kt_0)\ar[r]&\pi_1(\kt_i)\ar[r]&\pi_1({\rm Gr_i}).\\
&\cong \ZZ&\cong\ZZ/(dN/n)\ZZ&\cong0&\cong0}$$

$\bullet$ From the Leray spectral sequence one obtains the exact sequence \TBC{There is something to say here: The left hand term should be $H^1({\rm Gr}_i(N,S^d_iV),p_\ast\GG_m)$ which should just be $H^1({\rm Gr}_i(N,S^d_iV),\GG_m)$.}
$$\xymatrix@R=0.1pt{\cdots\ar[r]&\Pic({\rm Gr}_i)\ar[r]&\Pic(\kt_i)\ar[r]&H^0({\rm Gr}_i,R^1p_{i\ast}\GG_m)\ar[r]&H^2({\rm Gr}_i,p_{i\ast}\GG_m).\\
&\cong\ZZ&\cong \ZZ&\cong H^0({\rm Gr}_i,\ZZ/(dN/n)\ZZ)\\
&&&\cong \ZZ/(dN/n)\ZZ}$$
It is not difficult to see that the restriction map $\Pic(\kt_i)\to\Pic(\kt_0)$ is surjective. This shows that the generator $\ko(1)\in \Pic({\rm Gr}_i)$, the Pl\"ucker line bundle, pulls back to the $\ell\coloneqq(dN/n)$-th multiple of the 
ample generator $L_i$ of $\Pic(\kt_i)\cong\ZZ$.\smallskip

Assume $0\ne s\in H^0({\rm Gr}_i,\ko(1))$ is a section with its zero locus $Z(s)\subset{\rm Gr}_i$, a Schubert cycle of codimension one. Then $p_i^\ast s\in H^0(\kt_i,L_i^\ell=p_i^\ast \ko(1))$ and the standard
 construction induces a cyclic covering $\widetilde\kt_i\to\kt_i$ of degree $\ell$ which
 is ramified along $p_i^{-1}Z(s)$. Restricted to any fibre of $\kt_i\to{\rm Gr}_i$ over a point in ${\rm Gr}_i\,\setminus\,Z(s)$, it gives back the universal cover of the fibre $\widetilde\kt_i|_{\kt_0}\cong\widetilde\kt_0\to\kt_0$,
 cf.\ Remark \ref{rem:unicover}.

\begin{remark}\label{rem:contr}
To be more specific, consider the Pl\"ucker embedding ${\rm Gr}_i\,\hookrightarrow \PP(\bigwedge^{N}S^d_iV)$ and let $s=s_1^\ast\wedge\cdots\wedge s_{N}^\ast\in H^0({\rm Gr},\ko(1))$,
where $s_1,\ldots,s_{N}\in S^dV$ is a fixed basis. Then the open subset $U\coloneqq {\rm Gr}_i\,\setminus\,Z(s)$ parametrises all subspaces $A\subset S^d_iV$  mapping isomorphically onto $S^dV$
under the projection $S^d_iV=S^dV\oplus k^{\oplus i}\twoheadrightarrow S^dV$. Clearly, $U$ is isomorphic to an affine space; it is one of the standard affine charts of the Grassmannian.
\smallskip

Over the open set $U$ the space $\kt_i$ contracts to $\kt_0$, for projecting to
$\PP(S^dV)$ defines a morphism
\begin{equation}\label{eqn:proj}
\kt_i\supset p_i^{-1}(U)\twoheadrightarrow \kt_0
\end{equation}
which is the identity on the fibre $\kt_0=p^{-1}([S^dV])$. The fibres of $p_i^{-1}(U)\to\kt_0$
are affine spaces and $\kp_i\to\kt_i$ restricted to $p_i^{-1}(U)$ is the pull-back of $\kp_0\to\kt_0$:
$$\xymatrix{\kp_i\ar[d]&\ar@{_(->}[l]~\kp_i|_{p_i^{-1}U}\ar@{}[dr]|\times\ar[d]\ar[r]&\kp_0\ar[d]\\
\kt_i\ar[d]_{p_i}&p_i^{-1}U\ar@{_(->}[l]\ar[d]\ar[r]&\kt_0\\
{\rm Gr}_i&\ar@{_(->}[l]U.&}$$
\end{remark}

Another description of the projection  (\ref{eqn:proj}) goes as follows. Observe that the quotient bundle $S^d_iV^\vee\otimes\ko_{\kt_i}\twoheadrightarrow \pi_{i\ast}\ko_{\pi_i}(d)$
trivialises naturally over the pre-image $p_i^{-1}(U)$. Then  (\ref{eqn:proj}) is the classifying map $\varphi\colon p_i^{-1}(U)\to \kt_0$ defined in \S\! \ref{sec:ClassmapsI} associated with the Brauer--Severi variety $\kp_i|_{p_i^{-1}(U)}\to p_i^{-1}(U)$. 
%%%%%%%%%%%%%
\subsection{Rational sections} The projection $p_i\colon\kt_i\twoheadrightarrow {\rm Gr}_i$
admits sections over the  standard open charts
$U\cong {\mathbb A}^{N\cdot i}\subset {\rm Gr}_i$. These are the fibres of (\ref{eqn:proj}), which
 can also be described more explicitly as follows.
\smallskip

Recall that a standard affine chart ${\mathbb A}^{N\cdot i}\cong U_{\!A_0}\subset {\rm Gr}_i$
is defined as the image of $$\Hom(A_0,S^d_iV/A_0)\cong{\mathbb A}^{N\cdot i}\,\hookrightarrow
{\rm Gr}_i,~\xi\mapsto ({\rm id}+\xi)(A_0),$$ where  $[A_0]\in {\rm Gr}_i$ is any point and a decomposition
$S^d_iV=A_0\oplus S^d_iV/A_0$ is chosen. Hence, 

\begin{equation}\label{eqn:sec}\left(\begin{matrix} 1&0\\\Hom(A_0,S^d_iV/A_0)&1\end{matrix}\right)\subset \GL(S^d_iV)
\end{equation} determines a section of the quotient
$\GL(S^d_iV)\twoheadrightarrow {\rm Gr}_i={\rm Gr}(N,S^d_iV)$ over the open chart $U_{\!A_0}\subset {\rm Gr}_i$.
The  image $\widetilde U_{\!A_0}\subset \kt_i$ of (\ref{eqn:sec})  under the projection $\GL(S^d_iV)\twoheadrightarrow \kt_i
\cong \GL(S^d_iV)\,/\,\GL(V,i)$  then determines a section of $\kt_i\to {\rm Gr}_i$ over the
open chart $U_{\!A_0}\subset {\rm Gr}(N,S^d_iV)$. We thus have proved the following result.

\begin{prop}\label{prop:covGr}
The base of the global universal Brauer--Severi variety $\kt_i$ is covered by affine
spaces of codimension $N^2-n$ $$\widetilde U_{\!A_0}={\mathbb A}^{N\cdot i}\subset \kt_i$$ that define sections of the projection
$\kt_i\twoheadrightarrow {\rm Gr}_i$ over the standard open charts $U_{\!A_0}\subset {\rm Gr}_i$.\qed
\end{prop}

Recall that the Brauer group of an affine space is trivial and so $\Br(\widetilde U_{\!A_0})$ is trivial. Hence,
any family of smooth subvarieties $Y_t\subset \kt_i$ such that generically $Y_t$ is contained in one of the $\widetilde U_{\!A_0}$, i.e.\ $\eta_{Y_t}\in \widetilde U_{\!A_0}$, has the property that the
restriction $\ZZ/d\ZZ\cong \Br(\kt_i)\to\Br(Y_t)$ is trivial.\smallskip

%As observed before, the rational section $\widetilde U\to U=U_{\!S^dV}$ through the distinguished %subspace $A_0=S^dV\subset S^d_iV$ is a fibre of the map $p_i^{-1}U\to B_0$ in (\ref{eqn:proj}).
%%ntroduced in the previous section.
%As $Q_i$ restricted to $p_i^{-1}(U)$ is the pull-back of $Q_0$ under the classifying map $\varphi\colon %p^{-1}(U)\to B_0$, this proves the weaker
%statement that the restriction map $\Br(B_i)=\langle{\pmb\alpha}_i\rangle\to \Br(\widetilde U)$ is trivial.

%%%%%%%%%%%%%%%%%
\subsection{Covering family}\label{sec:covfam}
There is another covering family of $\kt_i$, $i\geq0$, to which the generator ${\pmb \alpha}_i\in \Br(\kt_i)$ restricts trivially.\smallskip

By construction, $\pi_i\colon\kp_i\to\kt_i$ is the universal family of subschemes in $\PP(S^d_iV)$ that 
are projective equivalent to the image of the Veronese embedding
$\nu_{d,i}\colon \PP(V)\,\hookrightarrow\PP(S^dV)\,\hookrightarrow\PP(S^d_iV)$.
Hence, $\kp_i$ can be viewed as a closed subscheme  $$\kp_i\subset \kt_i\times\PP(S^d_iV)$$
and, due to the homogeneous nature of the construction, the fibres $ q_i^{-1}(x)$ of the induced projection $q_i\colon\kp_i\twoheadrightarrow \PP(S^d_iV)$ are all isomorphic to each other. They shall be considered
as closed subschemes $$\kt_{i,x}\coloneqq \pi_i(q_i^{-1}(x))\subset \kt_i.$$

Since ${\pmb \alpha}_i\in \Br(\kt_i)$ pulls back to the trivial class on $\kp_i$, all the restrictions
${\pmb \alpha}_i|_{\kt_{i,x}}\in\Br(\kt_{i,x})$ are trivial. Presumably,  the Brauer group of $\kt_{i,x}$ itself is not.\smallskip

In comparison to Proposition \ref{prop:covGr}, the covering family thus produced consists of closed subschemes of $\kt_i$.

\begin{prop}
The base $\kt_i$ of the global universal Brauer--Severi variety $\kp_i\to\kt_i$ admits an isotrivial covering family $$\bigcup_{x\in\PP(S^d_iV)}\kt_{i,x}=\kt_i$$
by smooth closed subschemes $\kt_{i,x}\subset\kt_i$, $x\in \PP(S^d_iV)$, of codimension $N_i-n$, such that the restriction map 
$$\langle{\pmb\alpha}_i\rangle\cong \Br(\kt_i)\to \Br(\kt_{i,x})$$ is trivial.\qed
\end{prop}

The situation looks different for the Picard group. This is the following result which was
 used in the proof of Proposition \ref{prop:piUniv}.

\begin{lem}\label{lem:PicBiBix} As before, we assume that $d$ divides $n$.
The restriction defines an isomorphism
$$\Pic(\kt_i)\congpf \Pic(\kt_{i,x}).$$
\end{lem}

\begin{proof} The arguments are very similar to the ones in the proof of Proposition \ref{prop:Picard}.
We will restrict to the case $i>0$, which is the one that was used in \S\! \ref{sec:periodindexUniv}.

First, as a homogeneous space we can describe $\kt_{i,x}$
as $$\kt_{i,x}\cong\GL(S^d_iV)_x\,/\, H_x.$$
Here, $H_x\subset \GL(S^d_iV)_x$ is the subgroup obtained
from $\GL(V,i)_x$, see (\ref{eqn:GLVx}), by dividing out by the action of $\mu_d$ on the first block.
The group $\GL(S^d_iV)_x$ is the stabiliser of the $\GL(S^d_iV)$-action on $\PP(S^d_iV)$ at a point $x\in\PP(S^dV)\subset\PP(S^d_iV)$.\smallskip

We use \cite{Popov} to describe $\Pic(\kt_{i,x})$ and to compare it to $\Pic(\kt_i)$ by the following commutative diagram of exact horizontal sequences
$$\xymatrix@R=5pt{\cong \ZZ\cdot\det_{S^d_iV}&\cong\ZZ\cdot\det_V\oplus\ZZ\cdot\det_{k^{\oplus i}}&\\
\chi(\GL(S^d_iV))\ar[r]\ar[dd]&\chi(H)\ar[r]\ar[dd]&\Pic(\kt_{i})\ar[r]\ar[dd]&0\\
&&&\\
\chi(\GL(S^d_iV)_x)\ar[r]&\chi(H_x)\ar[r]&\Pic(\kt_{i,x})\ar[r]&0.}$$
The vertical maps are the restriction for the character groups and the Picard group.
\smallskip

Writing out $\GL(S^d_iV)_x$ and $H_x$ as matrix groups, their character groups are computed as
$$\chi(\GL(S^d_iV)_x)\cong\ZZ\cdot {\det}_k\oplus\ZZ\cdot{\det}_{k^{\oplus N_i-1}}~\text{ and }~\chi(H_x)\cong\ZZ\cdot{\det}_k\oplus \ZZ\cdot {\det}_{k^{\oplus n-1}}\oplus\ZZ\cdot{\det}_{k^{\oplus i}}.$$
(If $d$ does not divide $n$, then the description of $\chi(H_x)$ changes. In this case, the restriction
$\Pic(\kt_i)\to\Pic(\kt_{i,x})$ will have a non-trivial finite cokernel, cf.\ the comment after Proposition
\ref{prop:Picard}.)\smallskip

Following  the maps between the character groups as in the proof of Proposition \ref{prop:Picard},
one eventually deduces that the restriction indeed yields $\ZZ\cong\Pic(\kt_i)\congpf\Pic(\kt_{i,x})$.
\end{proof}

%%%%%%%%%%%%%%%%%%%%%%%%%
\subsection{Cohomology}\label{sec:Coh}
Our next goal is to compute the cohomology of the universal Brauer--Severi 
varieties $\pi_i\colon\kp_i\to\kt_i$.  We will restrict to the case $k=\CC$. Since the Leray--Hirsch theorem computes the rational cohomology of $\kp_i$ in terms of that of $\kt_i$,  we will restrict to the latter. Here is the easiest case.

\begin{prop}
The cohomology of $B_0\cong\PGL(S^dV)/\PGL(V)$ is given by the graded exterior algebra
$$\xymatrix{H^*(B_0,\QQ) = \bigwedge^\ast\nolimits_{\QQ}\langle\xi_{2n+1}, \xi_{2n+3}, \dots, \xi_{2N-1}\rangle,}$$
where $n=\dim V$, $N=\dim S^dV$ and each $\xi_i$ is in degree $i$. In particular, the Poincar\'e polynomial of $B_0$ is given by
\[ P_{B_0}(t) = \sum_{k \ge 0} \dim_{\mathbb{Q}} H^k(B_0,\QQ) t^k = \prod_{j=n+1}^N (1 + t^{2j-1}). \]
\end{prop}

\begin{proof}
To compute the terms in question, one uses that $\PGL(n)$ is homotopic to its maximal connected compact subgroup $\mathrm{PU}(n)=\mathrm{PSU}(n)$, whose cohomology can be computed from that of its finite cover $\mathrm{SU}(n)$ \cite[Cor.\ 4D.3]{hatcher}, giving
\begin{equation}\label{eqn:cohPGL}
\xymatrix{ H^*(\PGL(n),\QQ) \cong \bigwedge^\ast\nolimits_{\QQ}\langle\xi_3, \xi_5, \dots, \xi_{2n-1}\rangle.}
\end{equation}

From the Leray--Hirsch Theorem, we deduce that
\begin{equation}\label{eqn:lerayhirsch}
 H^*(\PGL(N),\QQ)\cong H^*(B_0,\QQ)\otimes H^*(\PGL(n),\QQ)
 \end{equation}
as long as we can prove that the pull-back map
\begin{equation}\label{eqn:pullbackcoh}
\nu_{d}^*\colon H^*(\PGL(N),\QQ) \to H^*(\PGL(n),\QQ)
\end{equation} 
is surjective and sends $\xi_i\mapsto\xi_i$ after multiplication by a non-zero scalar.\smallskip 

To do this we argue with the induced map on topological classifying spaces with their universal principal $G$-bundle $\mathrm{E}G\to \mathrm{B}G$, and the associated vector bundles $\mathcal{V}$ and $ \mathcal{V}_d$ over them
\begin{equation*}
    \xymatrix{
        \mathcal{V}=\mathrm{E}\PGL(V)\times_{\PGL(V)}V\ar[d] & \mathcal{V}_d=\mathrm{E}\PGL(S^dV)\times_{\PGL(S^dV)} S^dV\ar[d] \\
        \mathrm{B}\PGL(V)\ar[r]^{\nu_d} & \mathrm{B}\PGL(S^dV).
    }
\end{equation*}
A theorem of Borel, cf.\
%\footnote{This should be Thoerem 13.1 in Borel's ``Sur La Cohomologie des Espaces Fibres Principaux et ...'' but it's maybe more involved to unravel from that.} (apply
 \cite[Thm.\,1.81\,\&\,1.86]{Felix}, applied to the maximal compact subgroup $\mathrm{PSU}(V)$,  gives that
\[ H^*(\mathrm{B}\PGL(V),\QQ) \cong \QQ[{\rm c}_2,{\rm  c}_3, \ldots, {\rm c}_n]. \]
Here, ${\rm c}_i={\rm c}_i(\mathcal{V})\in H^{2i}(\mathrm{B}\PGL(V),\QQ)$ is the $i$-th Chern class of the tautological vector bundle $\mathcal{V}$, cf.\ \cite[14.3]{MS}. Since $\nu_d$ is induced by the $d$-th symmetric power representation%\footnote{One could argue in more detail here, talking about the actions in the corresponding large diagram with principal $\mathrm{GL}$-bundles and their associated vectors bundles as the $\PGL$-case above.}
, we obtain
\[ \nu_d^*\mathcal{V}_d = S^d\mathcal{V} \] 
where $S^d\mathcal{V}$ is the $d$-th symmetric power of the vector bundle $\mathcal{V}$. The $i$-th Chern classes of the symmetric product can be computed using the splitting principle and it is not hard to see %\footnote{ie there is no way I want to try get the general coefficient out of this computation!!}
that
\[ {\rm c}_i(S^d\mathcal{V}) = e_{d,i}\cdot {\rm c}_i(\mathcal{V}) + \text{lower order terms}\]
where $e_{d,i}$ is a non-zero constant and the remaining terms are an expression in the smaller degree Chern classes. Hence, by induction in $i$, the map $\nu_d^*\colon H^*(\mathrm{B}\PGL(V),\QQ)\to H^*(\mathrm{B}\PGL(S^dV),\QQ)$ is surjective. 
Consider now the Serre spectral sequence for the fibration
\[ \PGL(V) \to \mathrm{E}\PGL(V) \to \mathrm{B}\PGL(V). \]
As in the proof of \cite[Thm.\  1.81]{Felix}, we see that going from the cohomology of $\PGL(V)$ to that of its classifying space sends a generator $\xi_{2i-1}$ to the Chern classe ${\rm c}_i$ with a shift in degree by one. In particular, surjectivity of the pull-back map on cohomology of classifying spaces implies 
the surjectivity of (\ref{eqn:pullbackcoh}). To conclude, note that (\ref{eqn:lerayhirsch}) in combination with (\ref{eqn:cohPGL}) implies the proposition.
\end{proof}

In addition, we remark that the generators $\xi_{2j-1}\in H^{2j-1}(\kt_0,\QQ)$ are of type $(j,j)$ and
thus of weight $2j$ with respect to the mixed Hodge structure on the smooth, quasi-projective variety $\kt_0$. In other words, the Hodge--Deligne polynomial of $\kt_0$ is
$$\mu(t,u,v)=\prod_{j=n+1}^N(1+t^{2j-1}u^jv^j).$$
This is a consequence of the corresponding statement for the mixed Hodge structure
of $\PGL$ cf.\ \cite[Thm.\ 9.1.5]{Deligne3}.% or \cite[Ex.\ 3.5]{FS}.
\smallskip

The passage from the cohomological description of $\kt_0$ to the one of the  higher $\kt_i$, $i>0$ is not difficult and provided by the following result which concludes the proof of Theorem \ref{thm2}.

\begin{cor} The cohomology of the higher universal Brauer--Severi varieties $\kt_i$, $i>0$, is given
by $$H^*(\kt_i,\QQ)\cong H^\ast(\kt_0,\QQ)\otimes H^\ast({\rm Gr}(N,S^d_iV),\QQ).$$
\end{cor}

\begin{proof}
The assertion is again a consequence of the Leray--Hirsch theorem, for which we need to show that 
the restriction maps $H^q(B_i,\QQ)\to H^q(B_0,\QQ)$ are surjective. In fact, since the restriction map
is multiplicative, it suffices to lift the generators $\xi_{2j-1}$, $j=n+1,\ldots,N$.\smallskip

For the latter we use the
Leray spectral sequence $E_2^{p,q}=H^p({\rm Gr}_i,R^q p_\ast\QQ)\Rightarrow H^{p+q}(B_i,\QQ)$
for the projection $p_i\colon B_i\to{\rm Gr}_i={\rm Gr}(N,S^d_iV)$.  Since ${\rm Gr}_i$ is simply connected, all the local systems are actually constant and, therefore, $E_2^{p,q}\cong H^p({\rm Gr}_i,\QQ)\otimes H^q(B_0,\QQ)$. In particular,
the restriction morphism is the natural map $E^q=H^q(B_i,\QQ)\twoheadrightarrow E_\infty^{0,q}\subset E_2^{0,q}\cong H^q(B_0,\QQ)$. %Note that since $H^{\rm odd}({\rm Gr}_i,\QQ)=0$, all $E_r^{{\rm odd},\ast}$ vanish and, therefore, $d_r^{0,q}\colon E_r^{0,q}\to E_r^{r,q+1-r}$ is trivial for $r$ odd. 
\smallskip

Suppose now that one of the generators $\xi_{2j-1}\in H^{2j-1}(B_0,\QQ)=E_2^{0,2j-1}$
cannot be lifted to a class on $\kt_i$. Then there exists an $r$ such that  $0\ne d_r^{0,2j-1}(\xi_{2j-1})\in E_r^{r,2j-r}$. %and this $r$ has to be  even.
However, on the $E_2$-page, $E_2^{r,2j-r}\cong H^r({\rm Gr}_i,\QQ)\otimes H^{2j-r}(\kt_0,\QQ)$ and according to the Hodge--Deligne polynomial the mixed
Hodge structure on the %even cohomology group
$H^{2j-r}(\kt_0,\QQ)$ is concentrated in weight
$\geq 2j-r+1$. Therefore, $H^r({\rm Gr}_i,\QQ)\otimes H^{2j-r}(\kt_0,\QQ)$ is concentrated in weight $\geq 2j+1$. Since $\xi_{2j-1}$ is of weight $2j$ and the Leray spectral sequence lives in the
category of mixed Hodge structures, this gives a contradiction.
\end{proof}

\begin{remark}\label{rem:obs} Assume $X\subset\bar X$ is a Zariski  open subset of a smooth projective
variety $\bar X$ and $P\to X$ is a Brauer--Severi variety. Then for $i\gg0$ the classifying morphism
$\psi\colon X\to \kt_i$, see Proposition \ref{prop:GlobalUniv}, leads to classes $\psi^\ast\xi_{2j-1}\in H^{2j-1}(X,\QQ)$. If $P\to X$ extends
to a Brauer--Severi variety $\bar P\to \bar X$, then we may assume that $\psi$ is the restriction of 
a classifying morphism $\bar\psi\colon \bar X\to\kt_i$ associated with $\bar P$. As the Hodge structure of $\bar X$ is pure,
the classes $\bar\psi^\ast\xi_{2j-1}\in H^{2j-1}(\bar X,\QQ)$ are all trivial and so are their restrictions
$\psi^\ast\xi_{2j-1}\in H^{2j-1}(X,\QQ)$. Equivalently, if any of the classes $\psi^\ast\xi_{2j-1}$
is not zero, then $P\to X$  cannot be extend to a Brauer--Severi variety $\bar P\to \bar X$.\smallskip

 Clearly, if the class $\alpha\in \Br(X)$ of $P\to X$ is ramified, then an extension $\bar P\to \bar X$ cannot exist. However, it could happen
that $\bar P$ does not exist, e.g.\ because of $\psi^\ast\xi_{2j-1}\ne0$, but still its Brauer class $\alpha$ is unramified 
(e.g.\ because it can be represented by a different Brauer--Severi variety $P'\to X$ that does extend) and so is contained in $\Br(\bar X)\subset \Br(X)$. In other words, the classes $\psi^\ast\xi_{2j-1}$ are invariants of the Brauer--Severi variety
and not merely of its Brauer class. \smallskip

Note that the choices involved in the definition of the classifying
morphism $\psi$, cf.\ Remark \ref{rem:choices}, are all continuous and so the cohomology
classes $\psi^\ast\xi_{2j-1}\in H^\ast(X,\QQ)$ are in fact independent of these choices.
\end{remark}

Eventually, one can pass to the stable cohomology or, equivalently, the cohomology of
$\kt_\infty(d,n)\coloneqq\bigcup_{i\geq0} \kt_i(d,n)$. As it sits over ${\rm Gr}(N,\infty)=\bigcup_{i\geq0} {\rm Gr}(N,S^d_iV)$
and the cohomology of the latter is well known, see e.g.\ \cite[Thm.\ 4.D.4]{hatcher}, one obtains the following.

\begin{prop}\label{prop:cohBinfty}
For fixed $n$ and $d$ the rational cohomology ring of $\kt_\infty(d,n)$ is
$$\xymatrix{H^\ast(\kt_\infty(d,n),\QQ)\cong \bigwedge^\ast\nolimits_{\QQ}\langle\xi_{2n+1}, \xi_{2n+3}, \dots, \xi_{2N-1}\rangle}\otimes \QQ[{\rm c}_1,{\rm c}_2,\ldots,{\rm c}_{N}],$$
where as before $N=\dim S^dV$ with $\dim(V)=n$.\qed
\end{prop}

\begin{remark}\label{rem:class}
Recall that ${\rm Gr}(N,\infty)$ is homotopy equivalent to the classifying space ${\rm BGL}(N)$, which allows one
to compute the cohomology of the classifying space of the group $\GL(N)$ as $H^\ast({\rm BGL}(N),\QQ)\cong\QQ[{\rm c}_1,{\rm c}_2,\ldots,{\rm c}_N]$.\smallskip

 Similarly, $\kt_\infty$ is the classifying space
of the group $\GL(n)\,/\,\mu_d\cong\Aut(\PP^{n-1},\ko(d))$, see Remark \ref{rem:AutO(d)},
and so Proposition \ref{prop:cohBinfty} computes the cohomology of the classifying space
${\rm B}(\GL(n)\,/\,\mu_d)$
and the projection $\kt_\infty(d,n)\to{\rm Gr}(N,\infty)$ can be viewed as the map ${\rm B}(\GL(n)\,/\,\mu_d)\to
{\rm B}\GL(N)$ induced by the inclusion $ \GL(n)\,/\,\mu_d\,\hookrightarrow \GL(N)$.
\end{remark}
%%%%%%%%%%%%%%%%%%%%%

%%%%%%%%%%%%%%%%%%%%%%%%

\section{Discriminant avoidance (after de Jong and Starr)}\label{sec:dJS}
This section is devoted to the `discriminant avoidance' result of de Jong and Starr \cite{dJS}. It states that the period-index conjecture for all unramified
Brauer classes $\alpha\in \Br(X)$ on arbitrary projective varieties $X$ of fixed dimension implies that it
holds as well for all ramified classes $\alpha\in\Br(K(X))$.
More precisely, it is the following statement. (The smoothness assumption can be modified. For example, one can replace smooth by locally factorial.)

\begin{thm}[de Jong--Starr]\label{thm} Let $k$ be an algebraically closed field and let $m$ be  a fixed positive integer. Assume that there exists an integer $e_m$ such that 
\begin{equation}\label{eqn:conj:pi2}
\ind(\alpha)\mid\per(\alpha)^{e_m}
\end{equation} holds for all Brauer classes
$\alpha\in\Br(X)$ on all smooth projective varieties $X$ of dimension $\leq m$ over $k$.
Then {\rm (\ref{eqn:conj:pi2})} also holds for all $\alpha\in\Br(Y)$
on arbitrary  smooth quasi(!)-projective varieties $Y$ of dimension $\leq m$.
\end{thm}

The result is not stated explicitly in this  form in \cite{dJS}. In fact, it is there only used to treat the case of surfaces. However, the way the theory is set up by de Jong and Starr, it is  clearly designed to cover the general case and the article is rightly credited for the above theorem. The paper \cite{dJS} discusses other aspects that are less relevant to the period-index problem itself. We here focus on the main arguments and modify the geometric setup, shifting attention from Grassmannians to (universal) Brauer--Severi varieties. We
hope that  the main ideas  can be appreciated better from the more geometric perspective of the present article. In a recent paper by Reichstein and Scavia, one finds yet another approach, cf.\ \cite[Prop.\ 15.2]{RS2}.\smallskip

Note that in \cite[\S\! 7]{dJ} de Jong explains how to deduce the period--index conjecture for ramified classes on surfaces from the unramified case. The ideas there to reduce the ramified to the unramified situation at least morally also work in higher dimensions.\footnote{This was pointed out to us by O.\ Benoist.} Some care has to be applied in positive characteristic, for resolution of singularities is freely used in \cite{dJ}. The approach
in \cite{dJS} and here is different. While  \cite{dJ} uses ramified covers specialising to a union of multiple
copies of the variety, the specialisation in \cite{dJS} and here is obtained by a Bertini type argument.
%%%%%

\subsection{Index under specialisation} Following de Jong and Starr \cite[Lem.\ 2.3.4]{dJS}, see also \cite[Ch.\ 4]{Starr}, we first show that rational sections specialise.

\begin{prop}\label{prop:ratsec}
Assume $\kx\to D$ is a flat quasi-projective family with integral fibres over a smooth
curve $D$ with a distinguished closed point $0\in D$ such that $\kx|_{D\,\setminus \,\{0\}}\to D\,\setminus\, \{0\}$ is projective. Furthermore, assume that  $\kg\to\kx$  is a flat and projective (proper is enough) morphism. Assume that for
$0\ne t\in D$ the morphism $\kg_t\to\kx_t$ admits a rational section.
Then, also
$\kg_0\to\kx_0$ admits a rational section.
\end{prop}

\begin{proof}
To simplify the exposition, we shall assume that the ground field $k$ is not only algebraically closed but also uncountable. It is easy to reduce to this case: If for some extension $K/k$ the base change 
$\kg_{0}\times_kK\to \kx_{0}\times_kK$ admits a rational section then so does $\kg_0\to\kx_0$,
as $k$ is algebraically closed.\smallskip

Next, we compactify both families to projective morphisms $\,\bar{\!\kx}\to D$ and $\,\bar{\!\kg}\to D$.
This typically will have the effect that $\kx_0$ is an open but non-dense subset of $\,\bar{\!\kx}_0$.
 In other words, $\,\bar{\!\kx}_0$ may have more than one irreducible component. The flatness of $\,\bar{\!\kx}\to D$ and $\,\bar{\!\kg}\to D$ still holds.\smallskip

For each $0\ne t\in D$ we pick a rational section
of $\kg_t\to\kx_t$. Its closure $\kz_t\subset\kg_t$ defines a point in the fibre of the
relative Hilbert scheme ${\rm Hilb}(\,\bar{\!\kg}/D)\to D$ over $t\in D$. As the latter has only countably many components, we may shrink $D$ and assume that all $\kz_t$ are contained in the same component
$H\twoheadrightarrow D$ of the Hilbert scheme.\smallskip

Then, after passing to a multi-section of the projection $H\to D$, we may assume that
there exists a relative
family $\kz'\to D$ of subschemes in the fibres $\,\bar{\!\kg}_t$. For $t\ne 0$, the fibre $\kz'_t\subset\kg_t=\,\bar{\!\kg}_t$ is contained
in the same component of the Hilbert scheme as $\kz_t\subset\kg_t$ and hence of relative
degree one over $\kx_t$. The same is then also true for $\kz'_0\to\,\bar{\!\kx}_0$ over the
irreducible component of $\,\bar{\!\kx}_0$ obtained as the closure of the integral open subset 
$\kx_0\subset\,\bar{\!\kx}_0$. Indeed, $\kz'_0\subset \,\bar{\!\kg}_0$ defines a rational section of $\,\bar{\!\kg}_0\to\,\bar{\!\kx}_0$ over $\kx_0$ if and only if its numerical intersection with the class of the fibre  is one and this condition is constant in flat families.
\end{proof}

The following consequence can be viewed as a geometric rephrasing of \cite[Lem.\ 6.2]{dJ}.

\begin{cor}\label{cor:spec}
Assume $P\to\kx$ is a Brauer--Severi variety on a smooth family
$\kx\to D$ with integral fibres over a smooth curve $D$ and assume $0\in D$ is a fixed closed point. If $\alpha_t=[P_{\kx_t}]\in\Br(\kx_t)$ denotes the Brauer class of the restriction $P_{\kx_t}\to \kx_t$, then $$\ind(\alpha_0)\mid\ind(\alpha_t)$$
for the very general $t\in D$.
\end{cor}

\begin{proof}
In the application, we will only be interested in the case that
all $\ind(\alpha_t)$ for $t\ne 0$ divide a fixed number $r$ and we will
restrict to this case. 

Consider the relative Grassmannian $\kg\coloneqq {\rm Gr}(r-1,P)\to \kx$.
Since for $t\ne 0$ we assume $\ind(\alpha_t)\mid r$,  the projection
$\kg_t\to\kx_t$ admits a rational section. Then, Proposition \ref{prop:ratsec}
implies that also $\kg_0={\rm Gr}(r-1,P_{\kx_0})$ admits a rational section
and, therefore, $\ind(\alpha_0)\mid r$.
\end{proof}

%%%%%%%%%%%%%%%%
\subsection{Unramifying via deformation} Assume $P\to X$ is a Brauer--Severi variety  of relative dimension $n-1$ over a quasi-projective variety $X$. According to Lemma \ref{lem:Or}, there exists a relative $\ko(d)$ on $P$, where $d$ is the period of $[P]\in\Br(X)$. We then consider 
a vector space $V$ of dimension $n$ and fix an integer $i\geq0$  sufficiently large, depending 
on $n$, $d$, and $\dim(X)$.
% linear embedding $S^dV\,\hookrightarrow W$ into a vector space $W$ of sufficiently high dimension depending on $n$, $d$, and $\dim(X)$.
The classifying morphism according to Proposition \ref{prop:GlobalUniv} is denoted $\psi\colon X\to\kt_i$.\smallskip

The following is a version of \cite[Prop.\ 2.4.2]{dJS}, see also  \cite[Thm.\ 1.4]{RS1} where a similar situation for finite groups is considered.

\begin{prop}\label{prop:fam} Assume $\psi\colon X\to\kt_i$ is an embedding.
Then there exists a smooth family $\kx\to D$ over a smooth curve $D$ 
with a distinguished closed point $0\in D$ and a
Brauer--Severi variety $P_D\to\kx$ together with a relative $\ko(d)$ such that:
\begin{enumerate}
\item[(i)] The closed fibre $\kx_0$ over $0\in D$ is isomorphic to a non-empty open subset  $U\subset X$.
\item[(ii)] The restriction $P_D|_{\kx_0}$ to $\kx_0$  is isomorphic
to the restriction $P|_U\to U$.
\item[(iii)] The family $\kx^{\rm o}\to D\,\setminus\,\{0\}$ 
obtained by restricting $\kx\to D$  is projective.
\end{enumerate}
\end{prop}

\begin{proof}
According to Proposition \ref{prop:GlobalUniv}, the classifying morphism $\psi\colon X\to \kt_i$
has the property that $P$ is isomorphic to the pull-back of the global universal Brauer--Severi variety $\kp_i\to \kt_i$. We choose $i$ such that $\kt_i$ admits a projective compactification
$\bar\kt_i$ with a boundary $\partial\kt_i\subset\bar\kt_i$ of codimension $>\dim(X)$, cf.\ Proposition \ref{prop:Dim}. Concretely, it is enough to assume $i>\dim(X)-1$.\smallskip

Clearly, some high power $\ko(M)$ of an ample line bundle $\ko(1)$ on $\bar\kt_i$ admits global sections
$s_1,\ldots, s_\ell$, with $\ell\coloneqq{\rm codim}(X\subset\kt_i)=N_i\cdot N-\dim(X)$, such that $\psi(X)$ is an open subset of an irreducible component of their common zero set:
\begin{equation}\label{eqn:Inter}
\psi(X)\subset Z(s_1)\cap\ldots\cap Z(s_\ell)\subset\bar\kt_i.
\end{equation}

By assumption, $\psi$ is an embedding, i.e.\ $X\cong\psi(X)$. Thus,
after possibly increasing $M$, we can choose a one-dimensional family of sections $(s_{1,t},\ldots,s_{\ell,t})$, $t\in D$ with $(s_{1,0},\ldots,s_{\ell,0})$ satisfying (\ref{eqn:Inter})  and such that for $t\ne0$ the intersection
$ \kx_t\coloneqq Z(s_{1,t})\cap\ldots\cap Z(s_{\ell,t})\subset\bar \kt_i$ is smooth and not intersecting $\partial\kt_i$. \smallskip

In other words, we choose a generic quasi-projective curve $D$ in the Hilbert scheme of complete intersections in $\bar\kt_i$ which maps a distinguished point $0\in D$ to the point corresponding to $Z(s_1)\cap\ldots\cap Z(s_\ell)\subset\bar\kt_i$ and  such that $D\,\setminus\, \{0\}$
is mapped to the open subscheme of the Hilbert scheme parametrising only
smooth subvarieties contained in the open subset $\kt_i$.\smallskip

After taking out the irreducible components of $Z(s_{1,0})\cap\ldots\cap Z(s_{\ell,0})$ different
from the one containing $\psi(X)$, we obtain a family $\kx\to D$ satisfying (i) and (iii). The Brauer--Severi variety $P_D$ is
obtained by pulling back $\kp_i\to\kt_i$ under the natural projection $\kx\to \kt_i$.
\end{proof}
%%%%%%%
\subsection{Proof of Theorem \ref{thm}}
Assume  $P\to X$ is a Brauer--Severi variety over a smooth quasi-projective variety $X$ of dimension $m$. Replacing $X$ by its image $\psi(X)$ under the classifying map $\psi$, we can reduce to the case that $\psi$ is an embedding. According to Proposition  \ref{prop:fam}, there exists an extension to a Brauer--Severi variety $P_D$ on a smooth family
$\kx\to D$ with special fibre $\kx_0$ isomorphic to an open subset of $X$ and smooth projective fibres
$\kx_t$, $t\ne0$. Denote by $\alpha\in\Br(\kx)$ the Brauer class of $P_D$. Then
its restriction $\alpha_t\in\Br(\kx_t)$ is the class of $P_t\coloneqq P_D|_{\kx_t}$. \smallskip

If the period-index conjecture in the form of (\ref{eqn:conj:pi2}) is known  to hold for projective varieties, then we have $\ind(\alpha_t)\mid\per(\alpha_t)^{e_m}$ for $t\ne0$. 
Then, by virtue of Corollary \ref{cor:spec},  the index $\ind(\alpha_0)$ of $P$ on $X$ satisfies
$$\ind(\alpha_0)\mid\ind(\alpha_t)\mid \per(\alpha_t)^{e_m}.$$
%In principle, there are now two cases to distinguish: Either (i) $\per(\alpha_t)>\per(\alpha_0)$
%or (ii) $\per(\alpha_t)\leq \per(\alpha_0)$. Both cases can in principle occur. However,
By construction, each $P_t$ comes with a relative $\ko(d)$, where $d=\per(\alpha_0)$.
Hence, $\per(\alpha_t)\mid d=\per(\alpha_0)$ which implies the desired $\ind(\alpha_0)\mid\per(\alpha_0)^{e_m}$.\qed

%%%%%%%%%%%%%%%%%%%%%%%

\end{document}